\documentclass{amsart}
\usepackage{graphicx}
\usepackage{thmtools, thm-restate}
\usepackage{amssymb,latexsym}
\usepackage{a4wide}
\usepackage{psfrag}
\usepackage[all]{xy}
\usepackage{hyperref}

\hypersetup{hidelinks}

\let\cal\mathcal
\def\AA{{\cal A}}
\def\BB{{\cal B}}
\def\CC{{\cal C}}
\def\DD{{\cal D}}
\def\EE{{\cal E}}
\def\FF{{\cal F}}

\def\HH{{\cal H}}

\def\NN{{\cal N}}
\def\OO{{\cal O}}

\def\SS{{\cal S}}
\def\TT{{\cal T}}
\def\UU{{\cal U}}
\def\VV{{\cal V}}
\def\WW{{\cal W}}
\def\XX{{\cal X}}
\def\YY{{\cal Y}}

\let\blb\mathbb
\def\bC{{\blb C}}

\def\bX{{\blb X}}

\def\bZ{{\blb Z}}
\def\bS{{\blb S}}

\def\bN{{\blb N}}

\def\bZ{{\blb Z}}

\let\frak\mathfrak

\def\aa{\frak{a}}
\def\bb{\frak{b}}

\def\Mod{\operatorname{Mod}}

\def\mod{\operatorname{mod}}

\def\grmod{\operatorname{grmod}}

\def\coh{\mathop{\text{\upshape{coh}}}}

\def\rep{\operatorname{rep}}

\def\Ext{\operatorname {Ext}}

\def\Hom{\operatorname {Hom}}

\def\End{\operatorname {End}}
\def\RHom{\operatorname {RHom}}

\def\im{\operatorname {im}}

\def\ker{\operatorname {ker}}

\def\End{\operatorname {End}}

\def\add{\operatorname {add}}

\def\perf{\operatorname {perf}}

\DeclareMathOperator{\Proj}{Proj}

\DeclareMathOperator{\thick}{thick}
\DeclareMathOperator{\wide}{wide}

\newcommand\Db{D^{b}}

\renewcommand\t{\tau}

\newtheorem{lemma}{Lemma}[section]
\newtheorem{proposition}[lemma]{Proposition}
\newtheorem{theorem}[lemma]{Theorem}
\newtheorem{corollary}[lemma]{Corollary}

\newcounter{MyCounter}

\theoremstyle{definition}

\newtheorem{example}[lemma]{Example}
\newtheorem{definition}[lemma]{Definition}

\newtheorem{question}[lemma]{Question}

{

}

\theoremstyle{remark}

\newtheorem{remark}[lemma]{Remark}

\newdimen\uboxsep \uboxsep=1ex
\def\uboxn#1{\vtop to 0pt{\hrule height 0pt depth 0pt\vskip\uboxsep
\hbox to 0pt{\hss #1\hss}\vss}}

\def\uboxs#1{\vbox to 0pt{\vss\hbox to 0pt{\hss #1\hss}
\vskip\uboxsep\hrule height 0pt depth 0pt}}

\def\Ob{\operatorname{Ob}}

\newcommand\exa{\nopagebreak \begin{center}\smallskip \nopagebreak               \begin{minipage}[t]{6cm}\sloppy}
\newcommand\exb{\end{minipage}\kern 1cm\begin{minipage}[t]{8cm}\sloppy}
\newcommand\exc{\end{minipage}\kern -3cm \smallskip\end{center}}

\title{Derived equivalences for hereditary Artin algebras}
\author{Donald Stanley}
\address{Donald Stanley \\ Dept. of Math. \& Stats. \\ University of Regina \\ Regina, Canada S4S 4A5}\email{donald.stanley@uregina.ca}
\author{Adam-Christiaan van Roosmalen}
\address{Adam-Christiaan van Roosmalen \\ Universiteit Hasselt \\ Campus Diepenbeek \\ Departement WNI \\ 3590 Diepenbeek \\ Belgium}
\email{adamchristiaan.vanroosmalen@uhasselt.be}


\begin{document}
\date{\today}

\bibliographystyle{amsplain}

\begin{abstract}
We study the role of the Serre functor in the theory of derived equivalences.  Let $\AA$ be an abelian category and let $(\UU, \VV)$ be a $t$-structure on the bounded derived category $\Db \AA$ with heart $\HH$.  We investigate when the natural embedding $\HH \to \Db \AA$ can be extended to a triangle equivalence $\Db \HH \to \Db \AA$.  Our focus of study is the case where $\AA$ is the category of finite-dimensional modules over a finite-dimensional hereditary algebra.  In this case, we prove that such an extension exists if and only if the $t$-structure is bounded and the aisle $\UU$ of the $t$-structure is closed under the Serre functor.  
\end{abstract}

\maketitle

\tableofcontents

\section{Introduction}

Let $\AA$ be an abelian category.  In the bounded derived category $\Db \AA$, we consider the full subcategory $D^{\leq 0}_\AA$ of all objects $X$ such that $H^n X = 0$ for all $n > 0$, and the full subcategory $D^{\geq 0}_\AA$ of all objects $X$ such that $H^n X = 0$ for all $n < 0$.  We can recover $\AA$ (up to equivalence) as $D^{\leq 0}_\AA \cap D^{\geq 0}_\AA$.

A pair $(\UU, \VV)$ of full subcategories of $\Db \AA$ with properties similar to the pair $(D^{\leq 0}_\AA,D^{\geq 0}_\AA)$ given above is called a $t$-structure (see \cite{BeilinsonBernsteinDeligne82} or \S \ref{subsection:t-Structure}).  The definitions are chosen so that $\UU \cap \VV$ is an abelian category, called the \emph{heart} of $(\UU, \VV)$.

We will say that the abelian categories $\AA$ and $\BB$ are \emph{derived equivalent} if there is a triangle equivalence $F: \Db \AA \to \Db \BB$.  Using $F$, one can transfer the standard $t$-structure $(D^{\leq 0}_\BB,D^{\geq 0}_\BB)$ on $\Db \BB$ across to a $t$-structure $(\UU, \VV)$ on $\Db \AA$, whose heart $\UU \cap \VV$ is equivalent to $\BB$.

However, this situation is not representative for the general situation.  Indeed,  even though it is possible, for any $t$-structure $(\UU, \VV)$ on $\Db \AA$, to extend the natural embedding $\HH \to \Db \AA$ of the heart to a triangle functor $F:\Db \HH \to \Db \AA$, there might no choice of $F$ which is an equivalence.  If there is such a choice $F: \Db \HH \to \Db \AA$, we will say that the $t$-structure $(\UU, \VV)$ \emph{induces a derived equivalence} (see Definition \ref{definition:InducedDerivedEquivalence}).  We will discuss this in \S\ref{subsection:InducingDerivedEquivalences} (based on necessary and sufficient conditions given in \cite{BeilinsonBernsteinDeligne82}).

One necessary condition for a $t$-structure $(\UU, \VV)$ to induce a derived equivalence is that the $t$-structure must be \emph{bounded}, meaning that $\cup_{n \in \bZ} \UU[n] = \Db \AA = \cup_{n \in \bZ} \VV[n]$.  Here, we write $X[n]$ for the $n$-fold suspension of $X$.

To formulate a second necessary condition, we turn our attention to Serre functors (in the sense of \cite{BondalKapranov89}).  Thus, let $\CC$ be a $k$-linear ($k$ is a field) Hom-finite category; a Serre functor is an autoequivalence $\bS: \CC \to \CC$ together with automorphisms
$$\eta_{A,B}: \Hom(A,B) \cong \Hom(B,\bS A)^*,$$
for any $A,B \in \CC$, which are natural in $A,B$ and where $(-)^*$ is the $k$-dual.  Examples of categories which admit a Serre functor are the bounded derived category $\Db \coh \bX$ of coherent sheaves on $\bX$ (where $\bX$ is a smooth projective variety), and the bounded derived category $\Db \mod \Lambda$ of finite-dimensional right $\Lambda$-modules (where $\Lambda$ is a finite-dimensional algebra of finite global dimension).

Let $\CC$ be a triangulated category with a Serre functor $\bS$ and a $t$-structure $(\UU, \VV)$ with heart $\HH = \UU \cap \VV$.  We show in Corollary \ref{corollary:Needed} that if there is a triangle equivalence $F: \Db \HH \to \CC$ which extends the natural embedding $\HH \to \CC$, then $\bS \UU \subseteq \UU$.  Thus for a $t$-structure to induces a derived equivalence, it is necessary to satisfy some compatibility condition with the Serre functor.

One can now wonder in which cases the converse holds:

\begin{question}\label{question}
Let $\CC$ be a triangulated category which admits a Serre functor $\bS: \CC \to \CC$, and let $(\UU, \VV)$ be a $t$-structure on $\CC$ with heart $\HH = \UU \cap \VV$.  For which categories $\CC$ (and possibly for which restricted class of $t$-structures $(\UU, \VV)$ on $\CC$) are the following statements equivalent:
\begin{enumerate}
  \item the embedding $\HH \to \CC$ extends to a triangle equivalence $\Db \HH \to \CC$,
  \item the $t$-structure $(\UU, \VV)$ is bounded and $\bS \UU \subseteq \UU$?
\end{enumerate}
\end{question}

A first observation is that the existence of a triangle equivalence $\Db \HH \to \CC$ implies that $\CC$ is algebraic, in the sense of \cite{Keller06}.  Thus, to have any hope of a positive answer to Question \ref{question}, one needs to restrict oneself to algebraic triangulated categories.  In \S\ref{section:Proof}, we will give three examples of $t$-structures on (algebraic) triangulated categories where the answer to Question \ref{question} is negative.  We note that the triangulated category in Example \ref{example:NotInverse} is the bounded derived category of a hereditary category.

Before mentioning a case where one knows that the answer to Question \ref{question} is positive, we will introduce some notation.  Let $\Lambda$ be a finite-dimensional algebra (over a field $k$) with finite global dimension.  We will write $\mod \Lambda$ for the category of finite-dimensional right $\Lambda$-modules.  In this case, it is well-known that $\Db \mod \Lambda$ has Serre duality; we will denote the Serre functor by $\bS$, and we will denote $n$-fold suspension by $[n]$.  An object $E \in \Db \mod \Lambda$ is called a \emph{partial silting} object if $\Hom_{\Db \mod \Lambda}(E, E[i]) = 0$ for $i > 0$.  We will say that a $t$-structure $(\UU, \VV)$ is \emph{finitely generated} if there is a partial silting object $E \in \Db \mod \Lambda$ such that $\UU$ is the smallest full subcategory of $\Db \mod \Lambda$ closed under extensions and suspensions which contains $E$.  (Note that we do not require $\UU$ to be closed under desuspensions, thus the suspension $\Db \mod \Lambda \Db \mod \Lambda$ restricts to a functor $\UU \to \UU$ which is not necessarily an autoequivalence.  Put differently, $\UU$ is a suspended subcategory (\cite{KellerVossieck87}) of $\Db \mod \Lambda$ and not necessarily a triangulated subcategory.)

It follows from \cite[Lemma 4.6]{LiuVitoriaYang14} that one has a positive answer to Question \ref{question} for finitely generated $t$-structures on $\Db \mod \Lambda$.  We give a complete proof in \S\ref{subsection:FinitelyGenerated}.

Consequently, if all $t$-structures on $\Db \mod \Lambda$ were finitely generated, the answer to Question \ref{question} would be positive.  It was shown in \cite{BroomheadPauksztelloPloog13} that the class of derived discrete algebras (introduced in \cite{Vossieck01}) satisfies this property.  We will give a short account in \S\ref{section:DerivedDiscrete}.

The main result of this paper (Theorem \ref{theorem:MainTheorem} in the text) is that Question \ref{question} has a positive answer when $\Lambda$ is hereditary without any further restrictions on the $t$-structures one considers:

\begin{theorem}\label{theorem:Introduction}
Let $\Lambda$ be a finite-dimensional hereditary algebra over a field $k$, and let $\bS$ be the Serre functor in $\Db \mod \Lambda$.  Let $(\UU, \VV)$ be a $t$-structure on $\Db \mod \Lambda$ with heart $\HH = \UU \cap \VV$.  The embedding $\HH \to \Db \mod \Lambda$ can be extended to a triangle equivalence $\Db \HH \stackrel{\sim}{\rightarrow} \Db \mod \Lambda$ if and only if $(\UU, \VV)$ is bounded and $\bS \UU \subseteq \UU$.
\end{theorem}

The proof of the theorem is given in \S\ref{section:Proof}, but the main steps of the proof are given in \S\ref{section:Criterion}, \S\ref{section:NoExtProjectives}, and \S\ref{section:Reduction}.

\subsection*{Sketch of proof of Theorem \ref{theorem:Introduction}}

Let $\UU \subseteq \Db \mod \Lambda$ be a full subcategory.  We will say that an object $E \in \UU$ is Ext-\emph{projective} in $\UU$ if $\Hom(E,U[1]) = 0$ for all $U \in \UU$.  If $(\UU, \VV)$ is a $t$-structure, then any $E \in \UU$ which is Ext-projective is also a partial silting object in $\Db \mod \Lambda$.

In the proof of Theorem \ref{theorem:Introduction}, we consider two extremal cases: in the first case, $(\UU, \VV)$ is finitely generated (\S\ref{subsection:FinitelyGenerated}), while in the second case, $\UU$ has no nonzero Ext-projectives (\S\ref{section:NoExtProjectives}).

The proof of the former case is relatively straightforward; the latter case is more involved.  In Proposition \ref{proposition:NoProjectivesInAisle}, we construct a finitely generated $t$-structure $(\XX, \YY)$ on $\Db \mod \Lambda$ which is ``close enough'' to the given $t$-structure $(\UU, \VV)$;  more specifically:
\begin{enumerate}
\item $\bS \YY \subseteq \UU \subseteq \YY$, and
\item $\YY[1] \subseteq \UU$.
\end{enumerate}
It is then shown in Corollary \ref{corollary:Criterion} that the existence of such a $t$-structure $(\XX, \YY)$ implies that $(\UU, \VV)$ induces a derived equivalence.  The construction of the $t$-structure $(\XX, \YY)$ is heavily based on the description of $t$-structures on $\Db \mod \Lambda$ given in \cite{StanleyvanRoomalen12}; we will recall the relevant results in \S\ref{subsection:Classification}.

In the more general case where $\UU$ has nonzero Ext-projectives but where $(\UU, \VV)$ is not finitely generated, one cannot hope to construct a finitely generated $t$-structure $(\XX, \YY)$ as above (a counterexample is given in Example \ref{example:NoYY}).  To handle this case, we will use perpendicular categories to reduce the problem.  One attractive possibility is to take the right perpendicular category $E^\perp$ on an Ext-projective object $E \in \UU$.  It is shown in \cite{AssemSoutoTrepode08} that the subcategory $\UU \cap E^\perp$ is again an aisle in $E^\perp$, but in general it will not be closed under the Serre functor in $E^\perp$.

We will therefore take a slightly more subtle approach in \S\ref{section:Reduction}.  We show in Proposition \ref{proposition:SimpleTop} that each indecomposable Ext-projective object $E$ in $\UU$ corresponds to an indecomposable projective in the heart of the $t$-structure, and that $E$ has a simple top $S_E$ in the heart $\HH = \UU \cap \VV$.  We show in Proposition \ref{proposition:UUaisle} that $\UU \cap {}^\perp S_E$ is an aisle in ${}^\perp S_E$ which is closed under the Serre functor in ${}^\perp S_E$.

Furthermore, one can show (see Proposition \ref{proposition:HappelRickardSchofield}) that the category ${}^\perp S_E$ is triangle equivalent to $\Db \mod \Lambda'$ for a finite-dimensional hereditary algebra $\Lambda'$.  Proposition \ref{proposition:Reduction} then shows that the $t$-structure $(\UU, \VV)$ on $\Db \mod \Lambda$ induces a derived equivalence if and only if the induced $t$-structure $(\UU \cap {}^\perp S_E, \VV \cap {}^\perp S_E)$ on ${}^\perp S_E \cong \Db \mod \Lambda'$ induces a derived equivalence.  By applying this reduction finitely often, we reduce the problem to a known case.  

\smallskip

\textbf{Acknowledgments.}  The authors wish to thank Greg Stevenson and Michel Van den Bergh for many useful discussions, and Dong Yang and Steffen K{\"o}nig for sharing an early version of \cite{KonigYang14}. The authors also thank Jorge Vitoria for many useful comments on an early draft of this paper.

The second author expresses his gratitude for the support he received from the University of Regina.  This material is based upon work supported by the National Science Foundation under Grant No. 0932078 000, while the second author was in residence at the Mathematical Science Research Institute (MSRI) in Berkeley, California, during the spring semester of 2013.  This work was partially funded by the Eduard \v{C}ech Institute under grant GA CR P201/12/G028.

The second author is currently a postdoctoral researcher at FWO.
\section{Preliminaries and notation}

Throughout, we will fix a field $k$.  We will assume that all categories, functors, algebras, and vector spaces are $k$-linear.  Furthermore, we will assume that all categories are essentially small, i.e. equivalent to a category whose objects form a set.

When $\AA$ is an abelian category, we will write $\Db \AA$ for the bounded derived category.  The suspension functor is written by $[1]$, thus the $n^\text {th}$ suspension of $X \in \Db \AA$ is written as $X[n]$.  There is a fully faithful functor $\AA \to \Db \AA$, mapping every object in $\AA$ to a complex concentrated in degree zero.  When we interpret $\AA$ as a full subcategory in this way, we will write $\AA[0]$.

We write $\Ext^n_{\Db \AA}(A,B)$ for $\Hom_{\Db \AA}(A,B[n])$.  Since $\Ext^n_{\Db \AA}(A,B) \cong \Ext^n_\AA(A,B)$, naturally in $A,B \in \AA$, we often drop the subscript and thus write $\Ext^n(A,B)$.

A category is called \emph{Hom-finite} if $\dim_k \Hom(X,Y) < \infty$, for all objects $X,Y$.  An abelian category is \emph{Ext-finite} if $\dim_k \Ext^i(X,Y) < \infty$, for all objects $X,Y$ and all $i$.  We note that an abelian Ext-finite category is also Hom-finite.

For an algebra $\Lambda$, we will write $\Mod \Lambda$ for the category of right $\Lambda$-modules.  The full subcategory of finite-dimensional right $\Lambda$-modules is denoted by $\mod \Lambda$.

A full subcategory $\CC \subseteq \DD$ is called \emph{replete} if $\CC$ is closed under isomorphisms.  If $\CC$ is a replete subcategory of $\DD$ and the embedding has a left (right) adjoint, then we will say that $\CC$ is a \emph{reflective} (\emph{coreflective}) subcategory of $\DD$.  We will denote the left adjoint of the embedding $\CC \to \DD$ by $(-)^\CC$ and the right adjoint by $(-)_\CC$.

A full subcategory $\UU$ of a triangulated category $\CC$ is called a \emph{preaisle} if $\UU$ is closed under extensions and suspensions.  A coreflective preaisle is called an \emph{aisle}.

For an additive category $\CC$ with split idempotents, we will write $\add_\CC E$ for the smallest full subcategory of $\CC$ which contains $E$ and which is closed under direct summands and finite direct sums.  When there is no confusion about the ambient category, we will write $\add E$ for $\add_\CC E$.

\subsection{Hereditary categories}

Let $\AA$ be an abelian category.  We will say that $\AA$ is \emph{hereditary} if $\Ext_\AA^2(-,-) = 0$.  We will say that a finite-dimensional algebra $\Lambda$ is \emph{hereditary} if $\Mod \Lambda$ is a hereditary category.  In this case, $\mod \Lambda$ is a hereditary category as well.

When $\AA$ is a hereditary category, there is the following description of the objects in $\Db \AA$ (see for example \cite[5.2 Lemma]{Happel88}, \cite[Theorem 3.1]{Lenzing07}): every object $X \in \Db \AA$ is isomorphic to a direct sum of its cohomologies, thus
$$X \cong \oplus_{n \in \bZ} (H^n X)[-n].$$

\subsection{Serre duality}

Let $\CC$ be a Hom-finite category.  A \emph{Serre functor} on $\CC$ (in the sense of \cite{BondalKapranov89}) is an autoequivalence $\bS: \CC \to \CC$ together with automorphisms
$$\eta_{A,B}: \Hom(A,B) \cong \Hom(B,\bS A)^*,$$
for any $A,B \in \CC$, which are natural in $A,B$ and where $(-)^*$ is the vector space dual.  If $\CC$ is a triangulated category, then $\bS$ can be given the structure of a triangle equivalence (\cite[3.3 Proposition]{BondalKapranov89}).

Let $\AA$ be an Ext-finite abelian category.  As both $\AA$ and $\Db \AA$ are Hom-finite additive categories with split idempotents (idempotents split in $\AA$ because $\AA$ is abelian; that they split in $\Db \AA$ has been shown in \cite[2.10. Corollary]{BalmerSchlichting01}), they are Krull-Schmidt categories.

We say that $\AA$ has \emph{Serre duality} if $\Db \AA$ admits a Serre functor.  It has been shown in \cite{BondalKapranov89} that the following categories have Serre duality:
\begin{itemize}
  \item the category $\mod \Lambda$ of finite-dimensional modules over a finite-dimensional algebra $\Lambda$ with finite global dimension (in this case $\bS \cong - \stackrel{L}{\otimes}_k (\Lambda^*)$), and
  \item the category $\coh \bX$ of coherent sheaves on a smooth projective variety $\bX$ (in this case $\bS \cong \omega_\bX \stackrel{L}{\otimes}_{\OO_\bX} - [n]$ where $\omega_\bX$ is the dualizing sheaf and $n$ is the dimension of $\bX$).
\end{itemize}

It has been shown in \cite[Proposition 2.8]{Chen11} that $\Db \AA$ has a Serre functor if and only if $\Db \AA$ has Auslander-Reiten triangles (see \cite[Proposition I.2.3]{ReVdB02} for the case where $k$ is algebraically closed).  Writing $\tilde{\tau}$ for $\bS[-1]$, every indecomposable object $A \in \Db \AA$ admits an Auslander-Reiten triangle $\tilde{\tau} A \to M \to A \to \bS A$.

If $\AA$ is hereditary and has Serre duality, then it follows from \cite[Theorem I.3.3]{ReVdB02} that $\AA$ has Auslander-Reiten sequences.  We will denote the Auslander-Reiten translation in $\AA$ by $\tau$, thus for every nonprojective indecomposable object $A \in \AA$ there is an Auslander-Reiten sequence $0 \to \tau A \to M \to A \to 0$.

The Auslander-Reiten translation in $\AA$ and the functor $\tilde{\tau}: \Db \AA \to \Db \AA$ coincide on nonprojective indecomposable objects of $\AA$, meaning that for every indecomposable nonprojective $A \in \AA$, we have $(\tau A)[0] \cong \tilde{\tau} (A[0])$.  We will therefore write $\tau$ for $\tilde{\tau}$.  Furthermore, we will write $\tau^-$ for $\tau^{-1}$.

We will use the following result (see \cite[Lemma 1]{Keller05}).

\begin{proposition}\label{proposition:KellerSerre}
Let $\CC$ and $\DD$ be Hom-finite categories, and assume that $\CC$ has a Serre functor $\bS_\CC$.  Let $F:\DD \to \CC$ be a fully faithful functor.  If $F$ admits a right and a left adjoint $R,L: \CC \to \DD$, then $\DD$ has a Serre functor $\bS_\DD \cong R \circ \bS_\CC \circ F$.  A quasi-inverse to $\bS_\DD$ is $\bS_\DD^{-1} \cong  L \circ \bS_\CC^{-1} \circ F$.
\end{proposition}

\subsection{Wide and thick subcategories}

Let $\AA$ be an abelian category.  Following \cite{Hovey01}, we say that a full subcategory $\WW$ of $\AA$ is \emph{wide} if $\WW$ is closed under kernels, cokernels, and extensions in $\AA$.  For any subset of objects or any subcategory $\BB$ of $\AA$, we will write $\wide_\AA \BB$ (or just $\wide \BB$ if there is no confusion) for the wide closure of $\BB$ in $\AA$.

Note that a wide subcategory is abelian and closed under retracts.  If $\AA$ is hereditary, then any wide subcategory $\WW \subseteq \AA$ is also hereditary.  Indeed, since $\WW$ is closed under extensions, we know that $\Ext^1_\AA(W,-)|_\WW \cong \Ext^1_\WW(W,-)$, for all $W \in \WW$.  This shows that $\Ext^1_\WW(W,-)$ is right exact, and thus that $\WW$ is hereditary.

Let $\CC$ be a triangulated category.  A full triangulated subcategory $\DD$ of $\CC$ which is closed under retracts is called \emph{thick} in $\CC$.  For any subset of objects or any subcategory $\DD$ of $\CC$, we write $\thick_\CC \DD$ (or just $\thick \DD$ if there is no confusion) for the thick closure of $\DD$ in $\CC$.

We will be interested in the case where $\AA$ is the category $\mod \Lambda$ of finite-dimensional modules over a finite-dimensional hereditary algebra $\Lambda$.  In this case, we have the following well-known property.

\begin{proposition}\label{proposition:Wide}
Let $\Lambda$ be a finite-dimensional hereditary algebra.  Let $\WW \subseteq \mod \Lambda$ be a wide subcategory, and let $\Db_\WW \mod \Lambda$ be the full subcategory of $\Db \mod \Lambda$ consisting of all objects $X \in \Db \mod \Lambda$ such that $H^i(X) \in \WW$ for all $i \in \bZ$.  The following are equivalent.
\begin{enumerate}
\item $\WW$ is a reflective subcategory of $\mod \Lambda$,
\item $\WW$ is a coreflective subcategory of $\mod \Lambda$,
\item $\WW \cong \mod \Gamma$ for a finite-dimensional hereditary algebra $\Gamma$,
\item $\Db_\WW \mod \Lambda$ is a reflective subcategory of $\Db \mod \Lambda$,
\item $\Db_\WW \mod \Lambda$ is a coreflective subcategory of $\Db \mod \Lambda$.
\end{enumerate}
\end{proposition}

\begin{proof}
See \cite[Proposition 6.6 and Remark 6.7]{Krause12}.
\end{proof}

\subsection{Perpendicular subcategories and twist functors}\label{subsection:Perpendicular}

Let $\CC$ be a triangulated category, and let $\SS \subseteq \Ob \CC$.  We define the following full subcategories via their objects:
\begin{eqnarray*}
\Ob \SS^\perp &=& \{C \in \CC \mid \forall n \in \bZ: \Hom(\SS,C[n]) = 0\}, \\
\Ob {}^\perp \SS &=& \{C \in \CC \mid \forall n \in \bZ: \Hom(C, \SS[n]) = 0\},
\end{eqnarray*}
and for any $n \in \bZ$:
\begin{eqnarray*}
\Ob \SS^{\perp_n} &=& \{C \in \CC \mid \Hom(\SS,C[n]) = 0\}, \\
\Ob {}^{\perp_n} \SS &=& \{C \in \CC \mid  \Hom(C, \SS[n]) = 0\}.
\end{eqnarray*}

Note that, for any $\SS \subseteq \Ob \CC$, both $\SS^\perp$ and ${}^\perp \SS$ are thick subcategories.  When $S \in \Ob \Db \AA$, we write $S^\perp$ for $\{S\}^\perp$.  We will use similar definitions for ${}^\perp S, S^{\perp_n}$, and ${}^{\perp_n} S$.  For any set $\SS$, we have ${}^\perp \SS \cong {}^\perp \thick \SS$.

Let $\AA$ be an Ext-finite abelian category of finite global dimension, thus for any $X,Y \in \Db \AA$, we have $\sum_{i \in \bZ} \dim_k \Hom(X,Y[i]) < \infty$.  Let $S \in \Db \AA$ such that $\Hom(S,S[i]) = 0$ for all $i \not= 0$, and such that $A = \Hom(S,S)$ is semi-simple, then the thick subcategory $\thick_{\Db \AA} S$ generated by $S$ is equivalent to $\Db \mod A$ (see \cite{Keller94}, see also \cite[Theorem 5.1]{vanRoosmalen06}), and the embedding $\thick S \to \Db \AA$ has a left and a right adjoint, given on objects by
\begin{align*}
C_{\thick S} &\cong \RHom(S,C) \stackrel{L}{\otimes}_A S, \\
C^{\thick S} &\cong \RHom(C,S)^* \stackrel{L}{\otimes}_A S.
\end{align*}

\begin{remark}\label{remark:Evaluation}
The functor $-  \stackrel{L}{\otimes}_A S: \Db \mod A \to \Db \AA$ is left adjoint to $\RHom(S,-): \Db \AA \to \Db \mod A$ and thus, for every $C \in \Db \AA$, the co-unit of the adjunction yields a natural map $\RHom(S,C) \stackrel{L}{\otimes}_A S \to C$, called the \emph{evaluation map}.  Since we have assumed that $A$ is semi-simple, we have $\RHom(S,C) \stackrel{L}{\otimes}_A S \cong \bigoplus_{i \in \bZ} \Hom(S,C[i]) \otimes_A S[-i]$.  When there can be no confusion, we will just write $- \stackrel{L}{\otimes} -$ for $- \stackrel{L}{\otimes}_A -$.
\end{remark}

The embedding $\thick_{\Db \AA} S \to \Db \AA$ is part of a semi-orthogonal decomposition of $\Db \AA$, and it follows from \cite[Lemma 3.1]{Bondal89} (see also \cite[Lemma 3.1]{VandenBergh00}) that the embedding ${}^\perp S \to \Db \AA$ has a right adjoint $T^*_S: \Db \AA \to {}^\perp S$ and a left adjoint $T_S: \Db \AA \to {}^\perp S$.  For every $C \in \Db \AA$, there are triangles
$$T^*_S(C) \to C \to \RHom(C,S)^* \stackrel{L}{\otimes}_A S \to T^*_S(C)[1],$$
$$T_S(C)[-1] \to \RHom(S,C) \stackrel{L}{\otimes}_A S \to C \to T_S(C).$$

\begin{remark}
The functors $T_S$ and $T_S^*$ are the twist functors considered in \cite[\S 3.1]{vanRoosmalen12} (based on \cite{Seidel01}); this explains the notation of the functors.
\end{remark}

\begin{remark}\label{remark:KernelOfTwist}
Note that $T_S(X) = 0$ if and only if $T^*_S(X) = 0$ if and only if $X \in \thick S$.
\end{remark}

\begin{remark} We will be interested in the case where $\AA$ is the category $\mod \Lambda$ of finite-dimensional modules over a finite-dimensional hereditary algebra $\Lambda$, and where $S \in \Db \mod \Lambda$ is an exceptional object.  In this setting, it has been shown in \cite[Proposition 3]{HappelRickardSchofield88} that $^{\perp} S$ is equivalent to $\Db \mod \Lambda'$ where $\Lambda'$ is a finite-dimensional hereditary algebra with one fewer distinct simple (thus if $\mod \Lambda$ has $n$ isomorphism classes of simple objects, then $\mod \Lambda'$ has $n-1$ isomorphism classes of simple objects).
\end{remark}

We will also use perpendicular subcategories in the setting of abelian categories.  Thus, let $\AA$ be any abelian category, and consider a subset $\SS \subseteq \AA$.  We define the following full subcategories via their objects:
\begin{eqnarray*}
\Ob \SS^\perp &=& \{A \in \AA \mid \forall n \in \bZ: \Ext^n(\SS,A) = 0\}, \\
\Ob {}^\perp \SS &=& \{A \in \AA \mid \forall n \in \bZ: \Ext^n(A, \SS) = 0\}.
\end{eqnarray*}
In general, the categories $\SS^\perp$ and ${}^\perp \SS$ are not wide subcategories of $\AA$ and will not be abelian.  It follows from \cite[Proposition 1.1]{GeigleLenzing91} that $\SS^\perp$ and ${}^\perp \SS$ will be wide subcategories of $\AA$ if $\AA$ is hereditary.  Note that in this case, both $\SS^\perp$ and ${}^\perp \SS$ are abelian and hereditary.

We recall the following lemma from \cite[Lemma 3.6]{ReitenVandenBergh01}.

\begin{lemma}\label{lemma:ReitenVandenBergh}
Let $\AA$ be a hereditary category.  For any wide subcategory $\WW \subseteq \AA$, we have $\Db_{\WW^\perp}(\AA) = (\WW[0])^\perp = (\Db_{\WW}(\AA))^\perp$.
\end{lemma}

\begin{remark}
In the statement of Lemma \ref{lemma:ReitenVandenBergh}, the perpendicular $\WW^\perp$ is taken in $\AA$, while the perpendicular $(\WW[0])^\perp$ is taken in $\Db \AA$.
\end{remark} 

\begin{proposition}\label{proposition:WidePerpendicular}
Let $\Lambda$ be a finite-dimensional hereditary algebra and write $\AA$ for $\mod \Lambda$.  If $\WW \subseteq \AA$ is a wide subcategory satisfying the equivalent conditions of Proposition \ref{proposition:Wide}, then ${}^\perp \WW, \WW^\perp \subseteq \AA$ also satisfy the equivalent conditions of Proposition \ref{proposition:Wide}.
\end{proposition}

\begin{proof}
By assumption, the embedding $\Db_\WW \AA \to \Db \AA$ has a left adjoint and hence is part of a semi-orthogonal decomposition of $\Db \AA$.  It follows from \cite[Lemma 3.1]{Bondal89} that the embedding $(\Db_\WW \AA)^\perp \to \Db \AA$ has a right adjoint.  Lemma \ref{lemma:ReitenVandenBergh} shows that $\WW^\perp \subseteq \AA$ satisfies the equivalent conditions of Proposition \ref{proposition:Wide}.

The proof for ${}^\perp \WW$ is similar.
\end{proof}

\begin{proposition}\label{proposition:HappelRickardSchofield}
Let $\Lambda$ be a finite-dimensional hereditary algebra, and let $E \in \Db \mod \Lambda$ be an indecomposable object.  If $\Ext^1(E,E) = 0$, then ${}^\perp E \cong \Db \mod \Lambda'$ where $\Lambda'$ is a finite-dimensional hereditary algebra with one fewer distinct simple module than $\Lambda$ (thus if $\mod \Lambda$ has $n$ isomorphism classes of simple objects, then $\mod \Lambda'$ has $n-1$ isomorphism classes of simple objects).
\end{proposition}

\begin{proof}
Up to suspension, we may assume that $E \in \Db \mod \Lambda$ is isomorphic to a stalk complex concentrated in degree zero, thus $E \cong A[0]$ for some $A \in \mod \Lambda$.  It follows from \cite[Proposition 3]{HappelRickardSchofield88} that ${}^\perp A \subseteq \mod \Lambda$ is equivalent to $\mod \Lambda'$ for some finite-dimensional hereditary algebra $\Lambda'$ with one fewer distinct simple module than $\Lambda$, and it follows from Lemma \ref{lemma:ReitenVandenBergh} that ${}^\perp E \subseteq \Db \mod \Lambda$ is equivalent to $\Db_{{}^\perp A} \mod \Lambda$.  Finally, we have $\Db_{{}^\perp A} \mod \Lambda \cong \Db ({}^\perp A) \cong \Db \Lambda'$ since $\mod \Lambda$ is hereditary and thus every object in $\Db \mod \Lambda$ is isomorphic to a direct sum of its homologies.
\end{proof}

\subsection{Ext-projectives and silting subcategories}

In this subsection, we will consider a common generalization of projective objects.

\begin{definition}\label{definition:Extprojective}
Let $\AA$ be an abelian category and let $\CC$ be a full additive subcategory of $\AA$ or $\Db \AA$.  An object $E \in \CC$ is called \emph{Ext-projective} in $\CC$ or \emph{$\CC$-projective} if $\Ext^i(E,C) = 0$ for all $i > 0$ and all $C \in \CC$.  The \emph{category of $\CC$-projective objects} is the smallest full subcategory of $\CC$ containing all $\CC$-projective objects.
\end{definition}

\begin{remark} The category $\EE$ of $\CC$-projective objects is an additive category.
\end{remark}

\begin{remark} If $\UU \subseteq \Db \AA$ is closed under suspensions (thus $\UU[1] \subseteq \UU$, for example when $\UU$ is a preaisle), then the category of $\UU$-projective objects is $\UU \cap {}^{\perp_1} \UU$.
\end{remark}

\begin{definition}
Let $\EE \subseteq \Db \AA$ be a full additive subcategory.  We will say that $\EE$ is a \emph{partial silting subcategory} if $\Hom(E,F[n]) = 0$ for all $E,F \in \EE$ and all $n > 0$.  The partial silting subcategory $\EE$ is called a \emph{silting subcategory} if $\thick(\EE) = \Db \AA$.

An object $E \in \Db \AA$ is called a \emph{(partial) silting object} if the category $\add(E)$ is a (partial) silting subcategory.

Similarly, we will say that the subcategory $\EE$ is a partial tilting subcategory if $\Hom(E,F[n]) = 0$ for all $E,F \in \EE$ and all $n \not= 0$.  The partial tilting subcategory $\EE$ is called a \emph{tilting subcategory} if $\thick(\EE) = \Db \AA$.

An object $E \in \Db \AA$ is called a \emph{(partial) tilting object} if the category $\add(E)$ is a (partial) tilting subcategory.
\end{definition}

\begin{example}
For any full additive subcategory $\CC \subseteq \Db \AA$, the category of $\CC$-projectives is a partial silting subcategory, but not necessarily a partial tilting subcategory.
\end{example}

\begin{example}
Let $\Lambda$ be a finite-dimensional algebra.  Then $\Lambda[0] \in \Ob \Db \mod \Lambda$ is a partial tilting object.  It will be a tilting object if and only if $\Lambda$ has finite global dimension.  However, $\Lambda[0]$ will always be a tilting object in $K^b(\Proj \Lambda) \subseteq \Db \mod \Lambda$.
\end{example}

We will use the following property.

\begin{proposition}\label{proposition:HappelRingel}
Let $\AA$ be an Ext-finite hereditary abelian category, and let $A \in \Db \AA$ be an indecomposable object.  If $\Ext^1(A,A) = 0$, then $\End A$ is a skew field.
\end{proposition}

\begin{proof}
Since $\AA$ is hereditary, we know that $A \cong \oplus_{n \in \bZ} (H^n A)[-n]$, and since $A$ is indecomposable, there is an $n \in \bZ$ such that $A \cong (H^n A)[-n]$.  Hence, there is an $n \in \bZ$ such that $A \in \AA[-n]$.  The statement then follows from \cite[Lemma 4.1]{HappelRingel82} (see also \cite[Proposition 5.2]{Lenzing07}).
\end{proof}
\section{\texorpdfstring{Torsion pairs, weight structures, and $t$-structures}{Torsion pairs, weight structures, and t-structures}}

In this section, we recall the definitions and some properties of torsion pairs, weight structures, and $t$-structures that we will use in this article.

\subsection{Torsion pairs}

Let $\CC$ be any category and let $\XX \subseteq \CC$ be a full subcategory.  We will say that a map $f: X \to C$ is a \emph{right $\XX$-approximation} of $C \in \CC$ if $X \in \XX$ and any map $X' \to C$ (with $X' \in \XX$) factors through $f$.  Put differently, the map $\Hom(X',f) : \Hom(X',X) \to \Hom(X',C)$ is surjective.  If every object $C \in \CC$ has a right $\XX$-approximation, then we say that $\XX$ is a \emph{contravariantly finite} subcategory.

Let $\YY$ be a full subcategory of $\CC$.  Dually, we say that a map $g: C \to Y$ is a \emph{left $\YY$-approximation} of $C \in \CC$ if $Y \in \YY$ and any map $C \to Y'$ (with $Y' \in \YY$) factors through $g$.  This is equivalent to saying that map $\Hom(g,Y') : \Hom(Y,Y') \to \Hom(C,Y')$ induced by $g:C \to Y$ is surjective for all $Y' \in \YY$.  If every object $C \in \CC$ has a left $\YY$-approximation, then we say that $\YY$ is a \emph{covariantly finite} subcategory.

A subcategory $\XX \subseteq \CC$ which is both covariantly and contravariantly finite is called \emph{functorially finite}.

\begin{example}
If the embedding $\XX \to \CC$ has a right (left) adjoint, then $\XX$ is contravariantly (covariantly) finite.
\end{example}

Let $\CC$ be any triangulated category, and let $(\XX, \YY)$ be a pair of full subcategories which are closed under retracts, finite direct sums, and isomorphisms.  Following \cite{IyamaYoshino08}, we will say that $(\XX, \YY)$ is a \emph{torsion pair} in $\CC$ if $\Hom(\XX, \YY) = 0$ and if for any $C \in \CC$ there is a triangle $X \to C \to Y \to X[1]$ where $X \in \XX$ and $Y \in \YY$.  It is easy to check that the map $X \to C$ is a right $\XX$-approximation of $C$ and that the map $C \to Y$ is a left $\YY$-approximation of $C$.  In particular, $\XX$ is contravariantly finite and $\YY$ is covariantly finite.

\begin{remark}
A torsion pair is completely determined by a contravariantly finite subcategory $\XX \subseteq \CC$ which is closed under retracts and extensions.  We can then recover $\YY$ as $\XX^{\perp_0}$.
\end{remark}

\begin{remark}
Both $\XX$ and $\YY$ are closed under extensions.
\end{remark}

\subsection{Weight structures}

Weight structures were introduced by Bondarko (\cite{Bondarko10}) and Pauksztello (called co-$t$-structures, see \cite{Paukszello08}).  A pair $(\XX, \YY)$ of subcategories of $\CC$ is said to be a \emph{weight structure} on $\CC$ if $(\XX, \YY[1])$ is a torsion pair and $\XX \subseteq \XX[1]$.  Note that this implies that $\YY[1] \subseteq \YY$ and so, since $\YY$ is also closed under extension, $\YY$ is a preaisle.

Let $(\XX, \YY)$ be a weight structure on $\CC$.  For any object $C \in \CC$, there is a triangle
$$X \to C \to Y \to X[1]$$
where $Y \in \YY$ and $X \in \XX[-1]$.

A weight structure $(\XX, \YY)$ is \emph{bounded} \cite{Bondarko10} if
$$\bigcup_{n \in \bZ} \XX[n] = \CC = \bigcup_{n \in \bZ} \YY[n].$$
%
%\begin{proposition}\label{proposition:Weight}
%Let $(\XX,\YY)$ be a $t$-structure.  If $\XX$ is a finitely generated aisle, then $(({}^{\perp_0}\XX)[-1] , \XX)$ is a weight structure.
%\end{proposition}
%
%\begin{proof}
%
%\end{proof}

\subsection{\texorpdfstring{$t$-Structures}{t-Structures}}\label{subsection:t-Structure}

A pair $(\UU, \VV)$ is called a \emph{$t$-structure} if $(\UU, \VV[-1])$ is a torsion pair in $\CC$ and $\UU[1] \subseteq \UU$.

We know (\cite[Proposition 2.4]{BeligiannisReiten07}, see also \cite{StanleyvanRoomalen12} based on \cite[Proposition 1.4]{Neeman10}) that the embedding $\UU \to \CC$ admits a right adjoint and that the embedding $\VV \to \CC$ admits a left adjoint.

Recall that the torsion pair $(\UU, \VV[-1])$ is completely determined by $\UU \subseteq \CC$.  Thus, the $t$-structure $(\UU, \VV)$ is determined by the aisle $\UU$ (this is the point of view taken in \cite{KellerVossieck88}).

Alternatively, a $t$-structure (\cite{BeilinsonBernsteinDeligne82}) on $\CC$ can be defined as a pair $(\UU, \VV)$ of full subcategories of $\CC$ satisfying the following conditions:
\begin{enumerate}
\item $\UU[1] \subseteq \UU$ and $\VV \subseteq \VV[1]$,
\item $\Hom(\UU[1], \VV) = 0$,
\item $\forall C \in \CC$, there is a triangle $U \to C \to V \to U[1]$ with $U \in \UU$ and $V \in \VV[-1]$.
\end{enumerate}
Furthermore, we will say the $t$-structure is \emph{bounded} if
$$\bigcup_{n \in \bZ} \UU[n] = \CC = \bigcup_{n \in \bZ} \VV[n].$$
We will say that $\UU$ is \emph{bounded} if the associated $t$-structure $(\UU, \VV)$ is bounded.

The \emph{heart} of a $t$-structure $(\UU, \VV)$ is defined to be $\UU \cap \VV$.  We will denote the heart of the $t$-structure $(\UU, \VV)$ by $\HH(\UU)$, thus explicitly:
$$\HH(\UU) = \UU \cap (\UU^{\perp_{-1}}) = \UU \cap \VV.$$
It has been shown in \cite[Th\'{e}or\`{e}me 1.3.6]{BeilinsonBernsteinDeligne82} (see also \cite[Theorem IV.4.4]{GelfandManin03}) that the heart is an abelian category.

Let $\AA$ be an abelian category.  The \emph{standard $t$-structure} $(D^{\leq 0}, D^{\geq 0})$ on $\Db \AA$ is defined as:
\begin{eqnarray*}
\Ob D^{\leq 0} = \{X \in \Db \AA \mid \mbox{$H^n X = 0$ for $n > 0$}\}, \\
\Ob D^{\geq 0} = \{X \in \Db \AA \mid \mbox{$H^n X = 0$ for $n < 0$}\}.
\end{eqnarray*}
The heart $\HH$ of the standard $t$-structure is the full subcategory consisting of $X \in \Db \AA$ such that $H^n X = 0$ for $n \not= 0$.  Note that $\AA \cong \HH$.

\begin{remark}\label{remark:HeartDeterminesStructure}
A bounded $t$-structure $(\UU, \VV)$ on $\CC$ is completely determined by its heart $\HH$.  Indeed, we can recover $\UU$ and $\VV$ as the smallest full extension-closed subcategories, closed under suspension and desuspension, respectively.

For an abelian category $\AA$, the $t$-structure induced by the full embedding $\AA \to \Db \AA$, mapping an object to the stalk complex concentrated in degree zero, is the standard $t$-structure.
\end{remark}

Associated to a $t$-structure $(\UU, \VV)$, there are the following cohomological functors:
\begin{eqnarray*}
H^0_\UU(-) &=& \left((-)_{\UU}\right)^{\VV} = \left((-)^{\VV}\right)_{\UU}, \\
H^n_\UU(-) &=& H^0_\UU(-[n]).
\end{eqnarray*}
We refer to \cite{BeilinsonBernsteinDeligne82, GelfandManin03, Keller07} for more information.

\begin{remark}
When $(\UU, \VV)$ is the standard $t$-structure on $\Db \AA$, the cohomological functors $H^n_\UU(-)$ correspond to the usual functors $H^n(-)$.
\end{remark}

We will use the following lemma.

\begin{lemma}\label{lemma:ContainedHearts}
Let $(\UU, \VV)$ and $(\UU', \VV')$ be $t$-structures on a triangulated category $\CC$.  If $(\UU', \VV')$ is bounded and $\HH(\UU') \subseteq \HH(\UU)$ as subcategories of $\CC$, then $(\UU, \VV) = (\UU', \VV')$.
\end{lemma}

\begin{proof}
Since $(\UU', \VV')$ is bounded and $\HH(\UU') \subseteq \HH(\UU)$, we have $\UU' \subseteq \UU$ and $\VV' \subseteq \VV$ (see Remark \ref{remark:HeartDeterminesStructure}).  However, since $\UU = {}^{\perp_{-1}}\VV$ and $\UU' = {}^{\perp_{-1}}(\VV')$, we can use $\VV' \subseteq \VV$ to obtain that $\UU' \supseteq \UU$.  We find that $\UU' = \UU$ and thus also that $\VV' = \VV$, as required.
\end{proof}

We will now restrict ourselves to the case $\CC = \Db \mod \Lambda$ for a finite-dimensional algebra $\Lambda$.  A preaisle $\UU \subseteq \Db \mod \Lambda$ is said to be \emph{finitely generated} if there is a partial silting object $E \in \Db \mod \Lambda$ such that $\UU$ is the smallest preaisle containing $E$.  We will say that the preaisle $\UU$ is \emph{generated by $E$}.  Note that we can recover $\add E$ as $\UU \cap {}^{\perp_1} \UU$.

\begin{remark}\label{remark:FinitelyGenerated}
Let $E$ be a silting object in $\Db \mod \Lambda$, and let $\UU \subseteq \Db \mod \Lambda$ be the preaisle generated by $E$.  Following \cite[\S 2]{AiharaIyama12} and \cite[\S 5.4]{KonigYang14}, we have the following description: an object $A \in \Db \AA$ lies in $\UU$ if and only if $\Hom(E,A[n]) = 0$, for all $n > 0$.  See also \cite[Proposition 2.23]{AiharaIyama12}.
\end{remark}

%\begin{remark}\label{remark:FinitelyGenerated}
%Let $E$ be a silting object in $\Db \mod \Lambda$, and let $\UU \subseteq \Db \mod \Lambda$ be the preaisle generated by $E$.  One description of $\UU$ has been given in \cite[Proposition 2.23]{AiharaIyama12}.  In \cite[\S 5.4]{KonigYang14}, the following description has been given: an object $A \in \Db \AA$ lies in $\UU$ if and only if $\Hom(E,A[n]) = 0$, for all $n > 0$.
%\end{remark}

%\begin{proposition}
%Let $(\UU, \VV)$ be a $t$-structure on $\CC$ and assume that $\UU$ is finitely generated by a partial silting object $E$.  Then $E$ is a silting object if and only if $(\UU, \VV)$ is bounded.
%\end{proposition}
%
%\begin{proof}
%
%\end{proof}

We will use the following connection between weight structures and $t$-structures (\cite[Proposition 2.22]{AiharaIyama12} or \cite[Theorem 4.3.2]{Bondarko10}).

\begin{proposition}\label{proposition:Weight}
Let $\Lambda$ be a finite-dimensional algebra of finite global dimension.  Let $(\UU, \VV)$ be a bounded $t$-structure on $\Db \mod \Lambda$.  If $\UU$ is a finitely generated aisle, then $(\bS^{-1} \VV, \UU)$ is a bounded weight structure on $\Db \mod \Lambda$.
\end{proposition}

\subsection{Torsion theories}

Let $\AA$ be an abelian category.  A full additive coreflective subcategory $\TT \subseteq \AA$ is called a \emph{torsion class} if $\TT$ is closed under extensions and quotient objects.  We will say that a torsion class $\TT \subseteq \AA$ is a \emph{tilting torsion class} if every object $A \in \AA$ is a subobject of an object in $\TT$.

Dually, a full additive reflective subcategory $\FF \subseteq \AA$ is called a \emph{torsionfree class} if $\FF$ is closed under extensions and subobjects.  A torsionfree class $\FF \subseteq \AA$ is a \emph{cotilting torsionfree class} if every object $A \in \AA$ is a quotient object of an object in $\FF$.

\begin{remark}
If $\AA$ has enough injectives, then a torsion class $\TT \subseteq \AA$ is tilting if and only if $\TT$ contains the injective objects.
\end{remark}

A pair $(\TT, \FF)$ of full subcategories of $\AA$ is called a \emph{torsion theory} if $\Hom_\AA(\TT, \FF) = 0$ and for every object $A \in \AA$ there is a short exact sequence
$$0 \to T \to X \to F \to 0$$
where $T \in \TT$ and $F \in \FF$.  This implies that $\TT$ is a torsion class in $\AA$ and that $\FF$ is a torsionfree class in $\AA$.  We will say that the torsion theory is \emph{tilting} if $\TT$ is tilting, and that the torsion theory is \emph{cotilting} if $\FF$ is cotilting.

Following \cite[Proposition 2.1]{HappelReitenSmalo96} we can use a torsion theory $(\TT, \FF)$ on $\AA$ to specify a $t$-structure $(\UU, \VV)$ on $\Db \AA$:
\begin{eqnarray*}
\Ob \UU &=& \{A \in \Db \AA \mid H^{0} A \in \TT\, \text{ and } \forall n > 0: H^n A = 0 \}, \\
\Ob \VV &=& \{A \in \Db \AA \mid H^{-1} A \in \FF, \text{ and } \forall n < -1: H^n A = 0 \}.
\end{eqnarray*}

We will call the heart $\HH(\UU)$ of this $t$-structure the \emph{tilting} of $\AA$ by $(\TT, \FF)$.  It follows from \cite[Theorem I.3.3]{HappelReitenSmalo96} that $(\UU, \VV)$ induces an equivalence $\Db \HH(\UU) \to \Db \AA$ when $(\TT, \FF)$ is tilting or cotilting.

Let $\Lambda$ be a finite-dimensional algebra.   We will say that a torsion class $\TT \subseteq \mod \Lambda$ is \emph{finitely generated} when there is an object $G \in \mod \Lambda$ such that every object of $\TT$ is a quotient object of $G^{\oplus d}$, for some $d > 0$.  It follows from \cite[Corollaries VI.6.2]{AssemSimsonSkowronski06} that we can choose $G$ to be a $\TT$-projective.

%Let $\Lambda$ be a finite-dimensional algebra.   We will say that a torsion class $\TT \subseteq \mod \Lambda$ is \emph{finitely generated} when there is an object $G \in \mod \Lambda$ such that every object of $\TT$ is a quotient object of $G^{\oplus d}$, for some $d > 0$.  It follows from \cite[Corollaries VI.6.2 and VI.6.3]{AssemSimsonSkowronski06} that we can choose $G$ to be a $\TT$-projective, and thus $G$ is a partial tilting module in $\mod \Lambda$ (in the sense of \cite[Definition VI.2.1]{AssemSimsonSkowronski06}).

\subsection{\texorpdfstring{$t$-Structures for hereditary algebras}{t-Structures for hereditary algebras}} \label{subsection:Classification}

In this subsection, we will recall some results from \cite{StanleyvanRoomalen12}.  Let $\Lambda$ be a finite-dimensional hereditary algebra and let $\AA$ be the category $\mod \Lambda$ of finite-dimensional right $\Lambda$-modules.  Let $\UU \subseteq \Db \AA$ be any aisle.

By taking cohomologies, we can associate to $\UU$ a sequence $(H^{-n} \UU)_{n \in \bZ}$ of full subcategories of $\AA$ (here, the homologies are taken with respect to the standard aisle, not the aisle $\UU$).

We write $\NN(n)$ for $H^{-n} \UU \subseteq \AA$ and we write $\WW(n)$ for the wide closure of $H^{-n} \UU$ in $\AA$.  Since $\UU$ is closed under suspensions, we know that $\NN(n) \subseteq \NN(n+1)$ and thus also $\WW(n) \subseteq \WW(n+1)$, for all $n \in \bZ$.  Furthermore, it is shown in \cite{StanleyvanRoomalen12} that $\NN(n)$ is a tilting torsion class in $\WW(n)$, and in \cite[Proposition 7.3]{StanleyvanRoomalen12} that the embedding $\WW(n) \to \AA$ has a right adjoint (thus $\WW(n)$ satisfies the equivalent conditions of Proposition \ref{proposition:Wide}).

It is shown in \cite[Corollary 4.4]{StanleyvanRoomalen12} that $\WW(n-1) \subseteq \NN(n)$ and in \cite[Proposition 8.4]{StanleyvanRoomalen12} that $\NN(n) \cap {}^\perp \WW(n-1)$ is a tilting torsion class in $\WW(n) \cap {}^\perp \WW(n-1)$.  This leads to the following definition.

\begin{definition}
A \emph{refined $t$-sequence} is a pair $(\WW(-), t_\WW(-))$ where
\begin{itemize}
\item $\WW(-)$ is a poset map from $\bZ$ to the poset of wide coreflective subcategories in $\AA$, and
\item $t_\WW(-)$ is a map from $\bZ$ to the set of full replete subcategories of $\AA$ such that $t_\WW(n)$ is a tilting torsion class in $\WW(n) \cap {}^\perp \WW(n-1)$.
\end{itemize}
We denote by $\Delta(\AA)$ the set of all refined $t$-sequences of $\AA$.
\end{definition}

\begin{remark}\label{remark:FormOfPerpendiculars}
It follows from Propositions \ref{proposition:Wide} and \ref{proposition:WidePerpendicular} that $\WW(n) \cap {}^\perp \WW(n-1)$ is equivalent to $\mod \Gamma$ for a finite-dimensional hereditary algebra $\Gamma$.  
\end{remark}

Thus, given an aisle $\UU \subseteq \Db \AA$, we can associate a refined $t$-sequence by choosing $\WW(n)$ to be the wide closure of $H^{-n}(\UU)$ and by $t_\WW(n) = H^{-n}(\UU) \cap {}^\perp \WW(n-1)$.  The main theorem of \cite{StanleyvanRoomalen12} asserts that this map is a bijection.

\begin{theorem}\label{theorem:StanleyvanRoomalen}
Let $\AA$ be as above.  The above correspondence defines a bijection
$$\begin{array}{lcr}
\{\mbox{$t$-structures on $\Db \AA$}\}&
\stackrel{\sim}{\longleftrightarrow}&
\Delta(\AA).
\end{array}$$
\end{theorem}

\begin{remark}\label{remark:Construction}
There is a straightforward way to recover the aisle from the associated refined $t$-sequence: the aisle $\UU$ is the smallest preaisle in $\Db \AA$ containing the subcategories 
\begin{itemize}
\item $t_\WW(n)[n]$ for all $n \in \bZ$, and
\item $\WW(n)[n+1]$ for all $n \in \bZ$.
\end{itemize}

If $\cap_{n \in \bZ} \WW(n) = 0$, then one can describe $\UU$ as the smallest preaisle in $\Db \AA$ containing the subcategories $t_\WW(n)[n]$ for all $n \in \bZ$.
\end{remark}

The following proposition provides a convenient way to recover $H^{-n}(\UU) = \NN(n)$ from the associated refined $t$-sequence.

\begin{proposition}\label{proposition:RecoverFromSequence}
Let $\UU \subseteq \Db \AA$ be an aisle, and let $(\WW(-), t_\WW(-))$ be the associated refined $t$-sequence.  An object $X \in \AA$ lies in $\NN(n) \subseteq \AA$ if and only if there is a triangle
$$A[0] \to X \to B \to A[1]$$
where $A \in t_\WW(n) \subseteq \AA$ and $B \in \Db \WW(n-1) \subseteq \Db \AA$.
\end{proposition}

\begin{proof}
This follows directly from \cite[Proposition 8.8(1)]{StanleyvanRoomalen12}.
\end{proof}

We will use the following proposition.

\begin{proposition}\label{proposition:AislesInclusion}
Let $\UU, \UU' \subseteq \Db \AA$ be aisles and let $(\WW(-), t_\WW(-))$ and $(\WW'(-), t_{\WW'}(-))$ be the associated refined $t$-sequences.  If $\WW(n) = \WW'(n)$ for all $n \in \bZ$, then
\begin{itemize}
\item $\UU[1] \subseteq \UU'$, and
\item $\UU \subseteq \UU'$ if and only if $t_\WW(n) \subseteq t_{\WW'}(n)$, for all $n \in \bZ$.
\end{itemize}
\end{proposition}

\begin{proof}
This follows directly from Remark \ref{remark:Construction}.
\end{proof}

Given a refined $t$-sequence, one can use the following proposition to find the Ext-projectives in the corresponding aisle.

\begin{proposition}\label{proposition:FinitelyGenerated}
Let $\UU \subseteq \Db \AA$ be a bounded aisle and let $(\WW(-),t_\WW(-)) \in \Delta (\AA)$ be the associated refined $t$-sequence.  We write $\EE_n$ for the category of $t_\WW(n)$-projective objects.  The category of $\UU$-projective objects is $\oplus_{n \in \bZ} \EE_n[n]$.
\end{proposition}

\begin{proof}
Recall that $\AA$ is the category $\mod \Lambda$ of finite-dimensional modules over a finite-dimensional hereditary algebra $\Lambda$.  Since the aisle $\UU$ is bounded, one can recover $\UU$ from the associated refined $t$-sequence $(\WW(-),t_\WW(-))$ as the smallest preaisle in $\Db \AA$ containing the subcategories $t_\WW(n)[n]$ (see Remark \ref{remark:Construction}).

We will first show that every object in $\EE_n[n]$ is $\UU$-projective.  We know that $\EE_n[n] \subset \UU$, thus we need only to show that $\Hom(\EE_n[n], X_l[l+i]) = 0$, for all $i \geq 1$ and $X_l \in t_\WW(l)$ for all $l \in \bZ$.

If $l < n$, then it follows from $\EE_n \subseteq t_\WW(n) \subseteq \WW(n) \cap {}^\perp \WW(n-1)$ and $t_\WW(l) \subseteq \WW(l) \subseteq \WW(n-1)$ that $\Hom(\EE_n[n], X_l[l+i]) = 0$.

When $l=n$, this follows from the assumption that $\EE_n$ is the category of $t_\WW(n)$-projective objects.

When $l > n$, we have $l+i-n \geq 2$ so that
$$\Hom_{\Db \AA}(\EE_n[n], X_l[l+i]) \cong \Ext_\AA^{l+i-n}(\EE_n,X_l) = 0$$
since $\AA$ is hereditary.

Next, let $E \in \UU$ be $\UU$-projective; we want to show that $\oplus_{n \in \bZ} \EE_n[n]$.  Without loss of generality, we may assume that $E$ is indecomposable.  Assume that $E$ is concentrated in degree $-n$, thus $H^{i} E = 0$ when $i \not= -n$.  In particular, $E \in \NN(n)[n]$.  Recall that $\NN(n) = H^{-n}(\UU)$.

Since $E$ is $\UU$-projective, we know that $\Hom(E,\NN(n-1)[n+i]) = 0$ for all $i \geq 0$ (this uses that $\NN(n-1)[n-1] \subseteq \UU$).  Moreover, since $\NN(n-1)$ is a tilting torsion theory in $\WW(n-1)$, we know that $\Hom(E,\WW(n-1)[n]) = 0$.

Furthermore, since $E$ is a stalk complex in degree $-n$, we know that $\Hom(E,\WW(n-1)[n+i]) = 0$ for all $i < 0$.  Since $\AA$ is hereditary, we also have $\Hom(E,\WW(n-1)[n+i]) = 0$ for all $i \geq 2$.  Since $E$ is $\UU$-projective and $\WW(n-1)[n] \subseteq \NN(n)[n] \subseteq \UU$, we also have that $\Hom(E, \WW(n-1)[n+1])=0$.

This implies that $E \in {}^\perp \WW(n-1)$.  We conclude that $E \in \NN(n) \cap {}^\perp \WW(n-1) = t_\WW(n)$.

Finally, since $E$ is $\UU$-projective and $t_\WW(n)[n] \subset \UU$, we see that $E$ must be $t_\WW(n)[n]$-projective.
\end{proof}

\begin{corollary}\label{corollary:FinitelyGeneratedCriterion}
Let $\UU \subseteq \Db \AA$ be an aisle and let $(\WW(-),t_\WW(-)) \in \Delta (\AA)$ be the associated refined $t$-sequence.  The aisle $\UU \subseteq \Db \AA$ is finitely generated and bounded if and only if
\begin{itemize}
\item $\WW(n) = 0$ for $n \ll 0$, 
\item $\WW(n) = \AA$ for $n \gg 0$, and
\item $t_\WW(n) \subseteq \WW(n) \cap {}^\perp \WW(n-1)$ is finitely generated, for every $n \in \bZ$.
\end{itemize}
\end{corollary}

\begin{proof}
Recall that $\AA$ is the category $\mod \Lambda$ of finite-dimensional modules over a finite-dimensional hereditary algebra $\Lambda$.  The first two conditions are equivalent to saying that $\UU$ is bounded.  According to Proposition \ref{proposition:FinitelyGenerated}, the last condition is equivalent to saying that $\UU$ is finitely generated.
\end{proof}
\section{\texorpdfstring{$t$-Structures and derived equivalences}{t-Structures and derived equivalences}}

In this section, we recall some results from \cite{BeilinsonBernsteinDeligne82} about derived equivalences. 

\subsection{\texorpdfstring{$t$-Structures inducing derived equivalences}{t-Structures inducing derived equivalences}}\label{subsection:InducingDerivedEquivalences}

Let $\AA$ and $\BB$ be abelian categories.  We say that $\AA$ and $\BB$ are \emph{derived equivalent} if there is a triangle equivalence $\Db \BB \to \Db \AA$.  In this case, the standard $t$-structure $(D^{\leq 0}, D^{\geq 0})$ on $\Db \BB$ induces a $t$-structure $(\UU, \VV)$ on $\Db \AA$ with heart $\HH(\UU) \cong \BB$.  We capture this situation in the following definition.

\begin{definition}\label{definition:InducedDerivedEquivalence}
Let $\AA$ be an abelian category, and let $(\UU, \VV)$ be a $t$-structure on $\Db \AA$ with heart $\HH(\UU)$.  We will say that the $t$-structure $(\UU, \VV)$ (or just the aisle $\UU$) \emph{induces the derived equivalence} $\Db \HH(\UU) \cong \Db \AA$ if there exists a triangle equivalence $\Db \HH(\UU) \to \Db \AA$ which extends the natural embedding $\HH(\UU) \to \Db \AA$.
\end{definition}

\begin{remark}
The statement that a $t$-structure $(\UU, \VV)$ on $\Db \AA$ induces a derived equivalence is stronger than the statement that the heart $\HH(\UU)$ is derived equivalent to $\AA$.  Indeed, the former case requires the existence of a triangle equivalence $G: \Db \HH(\UU) \to \Db \AA$, while in the latter case we additionally require this functor to extend the natural embedding $\HH(\UU) \to \Db \AA$, meaning that the following diagram commutes up to natural equivalence:
$$\xymatrix{
\HH(\UU) \ar[r] \ar[d] & \Db \AA \\
\Db \HH(\UU) \ar@{-->}[ru]_{G}}$$
Example \ref{example:NoninducedDerivedEquivalence} below illustrates this difference.
\end{remark}

A triangle equivalence $\Db \HH(\UU) \to \Db \AA$ which extends the natural embedding $\HH(\UU) \to \Db \AA$, relates the standard $t$-structure on $\Db \HH(\UU)$ with the $t$-structure $(\UU, \VV)$ on $\Db \AA$, as in the following proposition. 

\begin{proposition}\label{proposition:ContainedHearts}
Let $\AA$ and $\BB$ be abelian categories and let $G: \Db \BB \to \Db \AA$ be any triangle equivalence.  Consider a $t$-structure $(\UU, \VV)$ on $\Db \AA$, and write $(\UU',\VV')$ for the $t$-structure on $\Db \AA$ induced from the standard $t$-structure $(D^{\leq 0}_\BB, D^{\geq 0}_\BB)$ on $\Db \BB$ by $G$, thus $\UU'$ and $\VV'$ are the essential images of $D^{\leq 0}_\BB$ and $D^{\geq 0}_\BB$ under $G$, respectively.

If $G(\BB[0]) \subseteq \HH(\UU)$, then $\UU = \UU'$ and $\VV = \VV'$.  
\end{proposition}

\begin{proof}
Directly from Lemma \ref{lemma:ContainedHearts}.
\end{proof}

Given an abelian category $\AA$ and a bounded $t$-structure $(\UU, \VV)$ on $\Db \AA$ with heart $\HH(\UU)$, we may extend the natural embedding $\HH(\UU) \to \Db \AA$ to a triangle functor $F: \Db \HH(\UU) \to \Db \AA$, called the \emph{realization functor} (see \cite[\S3.1]{BeilinsonBernsteinDeligne82}, \cite[3.2 Th\'{e}or\`{e}me]{KellerVossieck87}, or \cite[Appendix]{Beilinson89}).

\begin{remark}\label{remark:ExtInHeart}
In general, the realization functor $F$ is neither full nor faithful.  However, by \cite[Remarque 3.1.17(ii)]{BeilinsonBernsteinDeligne82}, we do have, for all $A,B \in \HH(\UU) \subset \Db \AA$:
\begin{eqnarray*}
\Hom_{\HH(\UU)}(A,B) &\cong& \Hom_{\Db \AA}(A,B), \\
\Ext^1_{\HH(\UU)}(A,B) &\cong& \Ext^1_{\Db \AA}(A,B).
\end{eqnarray*}
In particular, a short exact sequence $0 \to B \to M \to A \to 0$ in $\HH(\UU)$ corresponds to a triangle $B \to M \to A \to B[1]$ in $\Db \AA$ (see \cite[Th\'{e}or\`{e}me 1.3.6]{BeilinsonBernsteinDeligne82}).
\end{remark}

How far the realization functor $F:\Db \HH(\UU) \to \Db \AA$ is from being an equivalence, is determined by the relations between $\Ext_{\HH(\UU)}^n(A,B)$ and $\Ext_{\Db \AA}^n(A,B)$, for all $A,B \in \HH(\UU)$ and $n \geq 2$.  We will use the following equivalent properties.

\begin{proposition}\label{proposition:DerivedEquivalence}
Let $\AA$ be an abelian category, and let $(\UU, \VV)$ be a $t$-structure in $\Db \AA$.  The following conditions are equivalent:
\begin{enumerate}
\item\label{enumerate:DerivedEquivalence1} for all $A,B \in \HH(\UU)$, all $n \geq 2$ and every morphism $f:A \to B[n]$, there is a morphism $C \to A$ in $\HH(\UU)$ which is epic in $\HH(\UU)$, such that $C \to A \to B[n]$ is zero,
\item\label{enumerate:DerivedEquivalence2} for all $A,B \in \HH(\UU)$, all $n \geq 2$ and every morphism $f:A \to B[n]$, there is a morphism $B \to C$ in $\HH(\UU)$ which is monic in $\HH(\UU)$, such that $A \to B[n] \to C[n]$ is zero,
\item\label{enumerate:DerivedEquivalence3} for all $A,B \in \HH(\UU)$, all $n \geq 2$ and every morphism $f:A \to B[n]$, there is a sequence
$$A = A_0 \to A_1[1] \to A_2[2] \to \cdots \to A_n[n] = B[n]$$
of morphisms (with $A_i \in \HH(\UU)$) which composes to $f:A \to B[n]$.
\end{enumerate}
\end{proposition}

The proof of Proposition \ref{proposition:DerivedEquivalence} follows from the following lemma.  For future reference, it will be convenient to state this lemma separately.

\begin{lemma}\label{lemma:DerivedEquivalence}
Let $\AA$ be an abelian category, and let $(\UU, \VV)$ be a $t$-structure in $\Db \AA$.  Let $f:A \to B[n]$ be a morphism in $\Db \AA$ where $A,B \in \HH(\UU)$ and $n \geq 2$.  The following two conditions are equivalent:
\begin{enumerate}
\item there is an morphism $C \to A$ in $\HH(\UU)$ which is epic in $\HH(\UU)$, such that $C \to A \to B[n]$ is zero,
\item there is a sequence $A = A_0 \to A_1[1] \to B[n]$ of morphisms (with $A_1 \in \HH(\UU)$) which composes to $f:A \to B[n]$,
\end{enumerate}
and the following two conditions are equivalent:
\begin{enumerate}
\item there is a morphism $B \to C$ in $\HH(\UU)$ which is monic in $\HH(\UU)$, such that $A \to B[n] \to C[n]$ is zero,
\item there is a sequence $A = A_0 \to A_{n-1}[n-1] \to B[n]$ of morphisms (with $A_{n-1} \in \HH(\UU)$) which composes to $f:A \to B[n]$.
\end{enumerate}
\end{lemma}

\begin{proof}
We only show the equivalence of the first two conditions; the other equivalence is dual.  Thus assume that the first statement is true, and let $A_1 \cong \ker_{\HH(\UU)}(C \to A)$.  There is a triangle $A_1 \to C \to A \to A_1[1]$ in $\Db \AA$ (see Remark \ref{remark:ExtInHeart}).  Since the composition $C \to A \to B[n]$ is zero, we know that the map $f$ factors as $A \to A_1[1] \to B[n]$.

For the other direction, the triangle $A_1 \to C \to A \to A_1[1]$ in $\Db \AA$ built on the morphism $A \to A[1]$ corresponds to a short exact sequence $0 \to A_1 \to C \to A \to 0$ in $\HH(\UU)$ (see Remark \ref{remark:ExtInHeart}).  The epimorphism $C \to A$ is the morphism from the first statement.
\end{proof}

For the benefit of the reader, we present a proof of Proposition \ref{proposition:DerivedEquivalence}.

\begin{proof}[Proof of Proposition \ref{proposition:DerivedEquivalence}.]
Assume that (\ref{enumerate:DerivedEquivalence1}) holds, we want to show that (\ref{enumerate:DerivedEquivalence3}) holds.  When $n=2$, this follows directly from Lemma \ref{lemma:DerivedEquivalence}.

We will proceed by induction.  Consider a morphism $f:A \to B[n]$ where $A,B \in \HH(\UU)$ and $n \geq 3$.  By Lemma \ref{lemma:DerivedEquivalence}, we know that $f$ factors as
$$A = A_0 \stackrel{f_1}{\rightarrow} A_1[1] \stackrel{f_2}{\rightarrow} B[n]$$
(with $A_1 \in \HH(\UU)$).  By applying the induction hypothesis on morphism $f_2:A_1[1] \to B[n]$ (or, more accurately, on $f_2[-1]: A_1 \to B[n-1]$), we find a sequence
$$A = A_0 \stackrel{f_1}{\rightarrow} A_1[1] \to A_2[2] \to \cdots \to A_n[n] = B[n]$$
of morphisms (with $A_i \in \HH(\UU)$), composing to $f:A \to B[n]$, as required.

For the other direction, assume that (\ref{enumerate:DerivedEquivalence3}) holds.  To prove (\ref{enumerate:DerivedEquivalence1}), consider a morphism $f:A \to B[n]$ in $\Db \AA$ where $A,B \in \HH(\UU)$.  By (\ref{enumerate:DerivedEquivalence3}), we know there exists a sequence
$$A = A_0 \to A_1[1] \to A_2[2] \to \cdots \to A_n[n] = B[n]$$
of morphisms (with $A_i \in \HH(\UU)$), composing to $f:A \to B[n]$.  In particular, we have a sequence $A = A_0 \to A_1[1] \to B[n]$ of morphisms (with $A_1 \in \HH(\UU)$) which composes to $f:A \to B[n]$.  By Lemma \ref{lemma:DerivedEquivalence}, we know that there is a morphism $C \to A$ in $\HH(\UU)$, that is epic in $\HH(\UU)$, such that the composition $C \to A \to B[n]$ is zero.  This establishes (\ref{enumerate:DerivedEquivalence1}).

Similarly, one proves that (\ref{enumerate:DerivedEquivalence2}) and (\ref{enumerate:DerivedEquivalence3}) are equivalent.
\end{proof}

%\begin{proof}
%Assume that the first statement is true, and let $K \cong \ker(C \to A)$.  There is a triangle $K \to C \to A \to K[1]$.  Since the composition $C \to A \to B[n]$ is zero, we know that the map $f$ factors as $A \to K[1] \to B[n]$.  Iteration implies that the third statement holds.
%
%Assume that the third statement holds.  There is a triangle $A_1 \to C \to A \to A_1[1]$ in $\Db \AA$ and hence a short exact sequence $0 \to A_1 \to C \to A \to 0$ in $\HH(\UU)$.  The epimorphism $C \to A$ is the morphism from the first statement.
%
%Similarly, one shows that the second and the third statement are equivalent.
%\end{proof}

The following theorem is standard (see \cite[Proposition 3.1.16]{BeilinsonBernsteinDeligne82}, \cite[3.2 Th\'{e}or\`{e}me]{KellerVossieck87}, or \cite[Exercise IV.4.1]{GelfandManin03}).

\begin{theorem}\label{theorem:BBD}
Let $\AA$ be an abelian category, and let $(\UU, \VV)$ be a $t$-structure in $\Db \AA$.  The $t$-structure $(\UU, \VV)$ induces a derived equivalence $\Db \HH(\UU) \cong \Db \AA$ if and only if $(\UU, \VV)$ is a bounded $t$-structure and the equivalent conditions in Proposition \ref{proposition:DerivedEquivalence} hold.
\end{theorem}

\begin{proof}
If the $t$-structure is bounded and the equivalent conditions in Proposition \ref{proposition:DerivedEquivalence} hold, then \cite[Proposition 3.1.16]{BeilinsonBernsteinDeligne82} implies that the realization functor $F: \Db \HH(\UU) \to \Db \AA$ is an equivalence.  This proves one implication.

The other implication follows directly from Proposition \ref{proposition:ContainedHearts}.
%For the other implication, assume that the $t$-structure $(\UU, \VV)$ induces a derived equivalence, thus there is a triangle equivalence $G: \Db \HH(\UU) \cong \Db \AA$ lifting the natural embedding $\HH(\UU) \to  \Db \AA$.  We can use $G$ to transfer the natural $t$-structure $(D^{\leq 0}, D^{\geq 0})$ on $\Db \HH(\UU)$ to $\Db \AA$, obtaining a $t$-structure $(\UU', \VV')$ on $\Db \AA$ (thus $\UU'$ and $\VV'$ are the essential images of $D^{\leq 0}$ and $D^{\geq 0}$ under $G$, respectively).  It follows from Lemma \ref{lemma:ContainedHearts} that $(\UU, \VV) = (\UU', \VV')$.  One can use $G$ to prove that $(\UU, \VV)$ is a bounded $t$-structure and the equivalent conditions in Proposition \ref{proposition:DerivedEquivalence} hold, since these statements are true for the standard $t$-structure $(D^{\leq 0}, D^{\geq 0})$ on $\Db \HH(\UU)$.
\end{proof}

\begin{proposition}
Let $\AA$ be an abelian category, and let $(\UU, \VV)$ be a $t$-structure in $\Db \AA$ with heart $\HH(\UU) = \HH$.  The following are equivalent:
\begin{enumerate}
\item\label{enumerate:InducesDerivedEquivalence1} the $t$-structure $(\UU, \VV)$ induces a derived equivalence,
\item\label{enumerate:InducesDerivedEquivalence2} the realization functor $F: \Db \HH \to \Db \AA$ is an equivalence,
\item\label{enumerate:InducesDerivedEquivalence3} there is an equivalence $\Phi: \Db \HH \to \Db \AA$ such that $\UU$ and $\VV$ are the essential images of $D^{\leq 0}_\HH$ and $D^{\geq 0}_\HH$, respectively (here, $(D^{\leq 0}_\HH, D^{\geq 0}_\HH)$ is the standard $t$-structure on $\Db \HH$).
\end{enumerate}
\end{proposition}

\begin{proof}
We start by showing that (\ref{enumerate:InducesDerivedEquivalence1}) $\Leftrightarrow$ (\ref{enumerate:InducesDerivedEquivalence3}).  The implication (\ref{enumerate:InducesDerivedEquivalence1}) $\Rightarrow$ (\ref{enumerate:InducesDerivedEquivalence3}) follows directly from Proposition \ref{proposition:ContainedHearts}.  We now consider the direction (\ref{enumerate:InducesDerivedEquivalence3}) $\Leftarrow$ (\ref{enumerate:InducesDerivedEquivalence1}).  Since the composition $\HH \to \Db \HH \stackrel{\Phi}{\rightarrow} \Db \AA$ is fully faithful and the essential image coincides with the heart of $(\UU, \VV)$ on $\Db \AA$, there is an autoequivalence $\Psi: \HH \to \HH$ such that the following diagram essentially commutes:
$$\xymatrix{\HH \ar[r]^{\Psi} \ar[d] & \HH \ar[d] \\
{\Db \HH} \ar[r]_{\Phi} & {\Db \AA}}$$
Let $\Psi^{-1}: \HH \to \HH$ be a quasi-inverse to $\Psi$.  This then induces an equivalence $\Psi^{-1}: \Db \HH \to \Db \HH$ such that
$$\xymatrix{\HH \ar[r]^{\Psi^{-1}} \ar[d] & \HH \ar[d] \\
{\Db \HH} \ar[r]^{\Psi^{-1}} & {\Db \HH}}$$
essentially commutes.  Combining these two diagrams, we obtain the diagram
$$\xymatrix{\HH \ar[r]^{\Psi^{-1}} \ar[d] & \HH \ar[d] \ar[r]^{\Psi} & \HH \ar[d]\\
{\Db \HH} \ar[r]_{\Psi^{-1}} & {\Db \HH} \ar[r]_{\Phi} & {\Db \AA}}$$
where the squares essentially commute.  This establishes (\ref{enumerate:InducesDerivedEquivalence1}).

We will now show that (\ref{enumerate:InducesDerivedEquivalence1}) $\Rightarrow$ (\ref{enumerate:InducesDerivedEquivalence2}).  We can use a triangle equivalence $G: \Db \HH \to \Db \AA$ given by (\ref{enumerate:InducesDerivedEquivalence1}) to show that the equivalence conditions of Proposition \ref{proposition:DerivedEquivalence} hold (since they hold for the standard $t$-structure on $\Db \HH$ and $G$ extends the embedding $\HH \to \Db \AA$).  Theorem \ref{theorem:BBD} then establishes  (\ref{enumerate:InducesDerivedEquivalence2}).

Finally, since the realization functor extends the embedding $\HH \to \Db \AA$, the implication (\ref{enumerate:InducesDerivedEquivalence2}) $\Rightarrow$ (\ref{enumerate:InducesDerivedEquivalence1}) is direct.
\end{proof}

%\begin{proposition}\label{proposition:RealizationEquivalence}
%Let $\AA$ be an abelian category, and let $(\UU, \VV)$ be a $t$-structure in $\Db \AA$ with heart $\HH$.  If there exist a triangle equivalence $G: \Db \HH \to \Db \AA$ which extends the natural embedding $\HH \to \Db \AA$, then the realization functor $F:\Db \HH \to \Db \AA$ is an equivalence.
%\end{proposition}
%
%\begin{proof}
%Let $(D^{\leq 0}_\HH, D^{\geq 0}_\HH)$ be the standard $t$-structure on $\Db \HH$.  Note that $\UU$ and $\VV$ are the essential images of $D^{\leq 0}_\HH$ and $D^{\geq 0}_\HH$ under $G$, respectively.  Using that the equivalent conditions of Lemma \ref{lemma:DerivedEquivalence} hold for the standard $t$-structure $(D^{\leq 0}_\HH, D^{\geq 0}_\HH)$ on $\Db \HH$, we find that they also hold for the $t$-structure $(\UU, \VV)$ on $\Db \AA$.  Theorem \ref{theorem:BBD} now implies that the realization functor $F:\Db \HH \to \Db \AA$ is an equivalence.
%\end{proof}

\begin{remark}
As in \cite[\S 7]{Rickard89}, we do not make any claims about the uniqueness of a triangle equivalence $G: \Db \HH \to \Db \AA$ which extends $\HH \to \Db \AA$.  The question is addressed in \cite[\S 5.3]{Neeman91} (with an adjusted definition of triangle functors) and in \cite[Theorem 1.8]{MiyachiYekutieli01} for hereditary algebras.
\end{remark}

\begin{example}\label{example:NoninducedDerivedEquivalence}
Let $\grmod k$ be the category of finite-dimensional $\bN$-graded vector spaces.  We will write $\grmod_{> 0} k$ and $\grmod_0 k$ for the full subcategories given by all graded vector spaces $V$ such that $V_0 = 0$ and $V_i = 0$ for all $i > 0$, respectively.    On $\Db \grmod k$, we consider the $t$-structure $(\UU, \VV)$ given by:
\begin{eqnarray*}
\Ob \UU &=& \{X \in \Db \grmod k \mid \mbox{$H^{i} X = 0$ for $i > 0$ and $H^{i} X \in \grmod_{> 0} k$ for $i \leq 0$}\}, \\
\Ob \VV &=& \{X \in \Db \grmod k \mid \mbox{$H^{i} X \in \grmod_0 k$ for $i < 0$}\}.
\end{eqnarray*}
We can describe the heart $\HH = \UU \cap \VV$ as
$$\Ob \HH = \{X \in \Db \grmod k \mid \mbox{$H^i X = 0$ for $i \not= 0$, and $H^0 X \in \grmod_{> 0} k$}\}.$$

In particular, $\HH \cong \grmod_{> 0} k \cong \grmod k$.  Hence, $\grmod k$ and $\HH$ are equivalent categories, and thus derived equivalent.

However, we find that the realization functor $F: \Db \HH \to \Db \grmod k$ is not essentially surjective.  Indeed, an object $X \in \Db \AA$ lies in the essential image of $F$ if and only if $H^i X \in \grmod_{> 0}$ for all $i \in \bZ$; thus the essential image of $F$ is $\Db_{\grmod_{> 0}} \grmod k \subsetneq \Db \grmod k$.

This means that the $t$-structure $(\UU, \VV)$ on $\grmod k$ does not induce a derived equivalence, even though the categories $\grmod k$ and $\HH$ are derived equivalent.
\end{example}

\subsection{Compatibility with the Serre functor}

Let $\AA$ be an abelian category and let $(\UU, \VV)$ be a $t$-structure on $\Db \AA$.  Our next goal is to show that, if $\AA$ has Serre duality, it is necessary that $\bS \UU \subseteq \UU$ holds for $(\UU, \VV)$ to induce a triangle equivalence (see Corollary \ref{corollary:Needed} below).

\begin{lemma}\label{lemma:GoodEpimorphism}
Let $\AA$ be an abelian category, and write $(\UU, \VV)$ for the standard $t$-structure in $\Db \AA$.  Let $A \in \AA$ and $U \in \UU$.  For any morphism $A \to U[1]$, there is an epimorphism $B \to A$ in $\AA$ such that the composition $B[0] \to A[0] \to U[1]$ is zero.
\end{lemma}

\begin{proof}
Since $(\UU, \VV)$ is bounded, we can consider the maximal integer $n \in \bZ$ such that $U \in \UU[n]$ and the minimal integer $m \in \bZ$ such that $U \in \VV[m]$.  Note that $U$ can only be nonzero if $m \geq n \geq 0$.  We will assume that $U$ is nonzero, and proceed by induction on $m-n$.

When $m-n=0$, we know that $U \in \AA[m]$ and thus the morphism $A \to U[1]$ corresponds to an element in $\Ext^{m+1}(A,U[-m])$.  The existence of the required epimorphism $B \to A$ in $\AA$ is standard.

Assume now that $m-n > 0$.  Consider the triangle
$$(U[1])_{\UU[n+2]} \to U[1] \to H^{-n-1} (U[1])[n+1] \to (U[1])_{\UU[n+2]}[1].$$
Note that $H^{-n-1} (U[1])[n+1] \in \AA[n+1]$ and thus there exists an epimorphism $A_1 \to A$ in $\AA$ such that the composition $A_1 \to A \to H^{-n-1} (U[1])[n+1]$ is zero.  The composition $A_1 \to A \to U[1]$ thus factors through $(U[1])_{\UU[n+2]}$.  Note that 
$$(U[1])_{\UU[n+2]}[-1] \in \UU[n+1] \cap \VV[m]$$
so that the induction hypothesis implies the existence of an epimorphism $B \to A_1$ in $\AA$ such that $B \to A_1 \to (U[1])_{\UU[n+2]}[1]$ is zero.  The composition $B \to A_1 \to A$ is the required epimorphism.
\end{proof}

\begin{proposition}\label{proposition:OneWay}
Let $\AA$ be an Ext-finite abelian category with Serre duality, and write $(\UU, \VV)$ for the standard $t$-structure in $\Db \AA$.  Then the $t$-structure $(\UU, \VV)$ is bounded and $\bS \UU \subseteq \UU$.
\end{proposition}

\begin{proof}
It is clear that the standard $t$-structure $(\UU, \VV)$ is bounded.  We will show that $\bS \UU \subseteq \UU$. Seeking a contradiction, let $X \in \UU$ such that $\bS X \not\in \UU$.  Since $\UU$ is closed under suspensions and extensions, we may restrict ourselves to $X \in \AA[0]$.  Let $n \in \bZ$ be the largest integer such that $H^{n+1}(\bS X) \not= 0$.  Since $\bS X \not\in \UU$, we know that $n \geq 0$.  To ease notation, let $Y = X[n] \in \AA[n]$.  There is the triangle
$$(\bS Y)_\UU \to \bS Y \to (\bS Y)^{\VV[-1]} \to (\bS Y)_\UU[1].$$

We have chosen $Y$ in such a way that $(\bS Y)^{\VV[-1]} \cong H^{n+1}(\bS X)[-1] \in \AA[-1]$ is nonzero.  It follows from Lemma \ref{lemma:GoodEpimorphism} that there is an epimorphism $B \to (\bS Y)^{\VV[-1]}$ in $\AA[-1]$ such that the composition
$B \to (\bS Y)^{\VV[-1]} \to (\bS Y)_\UU[1]$
is zero.  Hence, $\Hom(B, \bS Y) \not= 0$ and by Serre duality, $\Hom(Y,B) \not= 0$.  However, the latter is impossible since $Y \in \AA[n]$ (with $n \geq 0$) and $B \in \AA[-1]$.  This finishes the proof.
\end{proof}

\begin{corollary}\label{corollary:Needed}
Let $\CC$ be a Hom-finite triangulated category with Serre duality and let $(\UU, \VV)$ be a $t$-structure in $\CC$.  If the embedding $\HH(\UU) \to \CC$ can be extended to a triangle equivalence $G: \Db \HH(\UU) \to \CC$, then $(\UU, \VV)$ is bounded and $\bS \UU \subseteq \UU$.
\end{corollary}

\begin{proof}
Choose a quasi-inverse $G^-$ to $G$ and let $\bS_\CC: \CC \to \CC$ be a Serre functor on $\CC$.  Note that $\Db \HH(\UU)$ has a Serre functor $\bS: \Db \HH(\UU) \to \Db \HH(\UU)$ given by $\bS = G^- \circ \bS_\CC \circ G$.

Using $G^-$, we may transfer the $t$-structure $(\UU, \VV)$ on $\CC$ to $\Db \HH(\UU)$; this gives a $t$-structure $(\UU', \VV')$ on $\Db \HH(\UU)$ where $\UU'$ and $\VV'$ are the essential images of $\UU$ and $\VV$ under $G^-$, respectively.

Since $G: \Db \HH(\UU) \to \CC$ extends the embedding $\HH(\UU) \to \CC$, we know that $G^-$ maps the heart $\HH(\UU)$ of the $t$-structure $(\UU, \VV)$ to $\HH(\UU)[0]$.  Hence, the essential image of $\HH(\UU)$ under $G^-$ coincides with $\HH(\UU)[0] \subset \Db \HH(\UU)$.  Since the former is the heart of the $t$-structure $(\UU', \VV')$ and the latter is the heart of the standard $t$-structure, we may apply Lemma \ref{lemma:ContainedHearts} to see that $(\UU', \VV')$ is the standard $t$-structure on $\Db \HH(\UU)$.

It now follows from Proposition \ref{proposition:OneWay} that $(\UU', \VV')$ is bounded and that $\bS \UU' \subseteq \UU'$.  Since Serre functors commute with equivalences, we may conclude that $(\UU, \VV)$ is bounded and that $\bS_\CC \UU \subseteq \UU.$
\end{proof}

The main goal of this paper is to show that in some cases, the converse of Corollary \ref{corollary:Needed} also holds (see Theorem \ref{theorem:Introduction}).
\section{Finitely generated aisles}\label{section:FinitelyGenerated}

Let $\Lambda$ be a finite-dimensional algebra of finite global dimension, so that $\Db \mod \Lambda$ has Serre duality.  The main result is Proposition \ref{proposition:FinitelyGeneratedAisle} where we show that the answer to Question \ref{question} is positive when one restricts oneself to finitely generated $t$-structures; thus, if $\UU$ is a finitely generated aisle in $\Db \mod \Lambda$, then $\UU$ induces a triangle equivalence $\Db \HH(\UU) \to \Db \mod \Lambda$ if and only if $(\UU, \VV)$ is bounded and $\bS \UU \subseteq \UU$. 

In \S\ref{section:DerivedDiscrete} we then consider the case where $\Lambda$ is derived discrete.  Here, one knows that all aisles are finitely generated, and hence Proposition \ref{proposition:FinitelyGeneratedAisle} implies that Question \ref{question} is answered positively for this class of algebras (see Corollary \ref{corollary:DerivedDiscrete}).

\subsection{Finitely generated aisles closed under the Serre functor}\label{subsection:FinitelyGenerated}

Recall that an aisle $\UU \subseteq \Db \mod \Lambda$ is called finitely generated if there is a partial silting object $E$ such that $\UU$ is the smallest preaisle containing $E$.  Let $(\UU, \VV)$ be a $t$-structure on $\Db \mod \Lambda$ and assume that $\UU$ is finitely generated. We will work toward Proposition \ref{proposition:FinitelyGeneratedAisle} where we show that $(\UU, \VV)$ induces a triangle equivalence $\Db \HH(\UU) \to \Db \mod \Lambda$ if and only if $(\UU, \VV)$ is bounded and $\bS \UU \subseteq \UU$.

\begin{proposition}\label{proposition:TiltingObject}
Let $\Lambda$ be a finite-dimensional algebra of finite global dimension.  Let $(\UU, \VV)$ be a $t$-structure in $\Db \mod \Lambda$ and assume that $\UU$ is finitely generated by a partial silting object $E$.
\begin{enumerate}
  \item The object $E$ is a silting object if and only if $(\UU, \VV)$ is a bounded $t$-structure.
  \item The object $E$ is a tilting object if and only if $(\UU, \VV)$ is a bounded $t$-structure and $\bS \UU \subseteq \UU$.
\end{enumerate}
\end{proposition}

\begin{proof}
\begin{enumerate}
  \item If $(\UU, \VV)$ is a bounded $t$-structure, then $\thick (E) = \thick(\UU) = \Db \mod \Lambda$ and hence $E$ is a silting object.  The other direction is \cite[Proposition 2.17]{AiharaIyama12}.
  
  \item This is essentially \cite[Lemma 4.6]{LiuVitoriaYang14}.  First, assume that $(\UU, \VV)$ is bounded and that $\bS \UU \subseteq \UU$.  From the first part of the proof, we know that $E$ is a silting object.  Furthermore, we have $\Hom(E, \bS E [n]) = 0$ for $n > 0$ (here we use that $\bS \UU \subseteq \UU$ and that $E \in \UU$ is $\UU$-projective).  By Serre duality, we have $\Hom(E,E[-n]) = 0$ for $n > 0$.  This shows that $E$ is a partial tilting object.  Since $(\UU, \VV)$ is bounded, the first part of the proof shows that $E$ is a tilting object.
  
  For the other direction, assume that $E$ is a tilting object.  The first part of the proof shows that $(\UU, \VV)$ is a bounded $t$-structure.  Since $E$ is a tilting object, we know that $\Hom(E,E[-n]) = 0$ for $n > 0$ and thus by Serre duality that $\Hom(E,\bS E[n]) = 0$.  Remark \ref{remark:FinitelyGenerated} shows that $\bS E \in \UU$.  Since $\bS \UU$ is the smallest preaisle in $\Db \mod \Lambda$ containing $\bS E$, we conclude that $\bS \UU \subseteq \UU$.
\end{enumerate}
\end{proof}

\begin{proposition}\label{proposition:FiniteGlobalDimension}
Let $\Lambda$ be a finite-dimensional algebra of finite global dimension.  If $E \in \Db \mod \Lambda$ is a tilting object, then $\End E$ has finite global dimension.
\end{proposition}

\begin{proof}
It follows from \cite{Keller94} that there is an equivalence $\perf (\End E) \to \Db \mod \Lambda$, and then \cite[Proposition 7.25]{Rouquier08} implies that $\End E$ has finite global dimension.
\end{proof}

\begin{remark}
If one replaced the finite-dimensional algebra $\Lambda$ in Proposition \ref{proposition:FiniteGlobalDimension} by a small category and consequently relaxed the ``tilting object'' to ``tilting subcategory,'' the statement would be false.  Indeed, in \cite[Example 4.11]{BergVanRoosmalen14} there is an example of a small category $\aa$ (such that $\mod \aa$ is Ext-finite, hereditary, and has Serre duality) and a tilting subcategory $\bb \subseteq \Db \mod \aa$ (thus $\Db \mod \aa \cong \Db \mod \bb$) such that nonetheless, $\mod \bb$ has infinite global dimension.
\end{remark}

The following example shows that Proposition \ref{proposition:FiniteGlobalDimension} does not hold when we only require $E$ to be a partial tilting module.

\begin{example}
Let $Q$ be the quiver
$\xymatrix@1{a \ar@/_/[r]_{\alpha} & b \ar@/_/[l]_{\beta}}$
and let $\Lambda = kQ / (\alpha \beta)$.  Note that $\Lambda$ is a finite-dimensional algebra of global dimension two. Let $P_a$ be the projective associated to the vertex $a$.  The object $P_a[0] \in \Db \mod \Lambda$ is a partial tilting object, but the endomorphism algebra $\End P_a \cong k[t]/(t^2)$ has infinite global dimension.
\end{example}

\begin{proposition}\label{proposition:FinitelyGeneratedAisle}
Let $\Lambda$ be a finite-dimensional algebra of finite global dimension.  Let $(\UU, \VV)$ be a $t$-structure in $\Db \mod \Lambda$.  If $\UU$ is finitely generated, then $(\UU, \VV)$ induces a triangle equivalence $\Db \HH(\UU) \to \Db \mod \Lambda$ if and only if $(\UU, \VV)$ is bounded and $\bS \UU \subseteq \UU$.
  In this case, $\HH(\UU) \cong \mod \Gamma$, for a finite-dimensional algebra $\Gamma$ of finite global dimension.
\end{proposition}

\begin{proof}
Assume first that $(\UU, \VV)$ induces a triangle equivalence $\Db \HH(\UU) \to \Db \mod \Lambda$.  Corollary \ref{corollary:Needed} shows that $(\UU, \VV)$ is a bounded $t$-structure and $\bS \UU \subseteq \UU$.

For the other direction, assume that $(\UU, \VV)$ is bounded and $\bS \UU \subseteq \UU$.  By Proposition \ref{proposition:TiltingObject} we know that $\UU$ is generated by a tilting object $E$, and by Proposition \ref{proposition:FiniteGlobalDimension} that the global dimension of $\Gamma = \End E$ is finite.  Following \cite{Rickard89}, the embedding $E \to \Db \mod \Lambda$ lifts to an equivalence $- \stackrel{L}{\otimes}_{\Gamma}{E}: \Db \mod \Gamma \to \Db \mod \Lambda$, which maps the standard $t$-structure on $\Db \mod \Gamma$ to the $t$-structure $(\UU, \VV)$ on $\Db \mod \Lambda$.  This shows that $(\UU, \VV)$ induces a triangle equivalence $\Db \HH(\UU) \to \Db \mod \Lambda$, where $\HH(\UU) \cong \mod \Gamma$.
\end{proof}

\subsection{Application: derived equivalences for derived discrete algebras}\label{section:DerivedDiscrete}

In this subsection, let $k$ be an algebraically closed field and let $\Lambda$ be a finite-dimensional algebra (in this subsection, we do not require $\Lambda$ to be hereditary).  Let $K_0(\mod \Lambda)$ be the Grothendieck group of $\mod \Lambda$.  For an object $M \in \mod \Lambda$, we will write $[M] \in K_0(\mod \Lambda)$ for the corresponding element in the Grothendieck group.  Let $H^i: \Db (\mod \Lambda) \to \mod \Lambda$ be the usual homology functors, thus the homology functors associated with the standard $t$-structure on $\Db (\mod \Lambda)$.  

The following definition is based on \cite{Vossieck01} (see \cite{BroomheadPauksztelloPloog13}).

\begin{definition}
The category $\Db (\mod \Lambda)$ is called \emph{discrete} if for every function $v: \bZ \to K_0(\mod \Lambda)$, there are only finitely many isomorphism classes $X \in \Db (\mod \Lambda)$ such that $[H^i(X)] = v(i) \in K_0(\mod \Lambda)$ for all $i \in \bZ$.  We will call such an algebra $\Lambda$ \emph{derived discrete}.
\end{definition}

\begin{remark}
If $\Lambda$ and $\Gamma$ are finite-dimensional algebras such that $\Db (\mod \Lambda) \cong \Db (\mod \Gamma)$, then $\Lambda$ is derived discrete if and only if $\Gamma$ is derived discrete.
\end{remark}

In \cite[2.1 Theorem]{Vossieck01}, the derived discrete algebras have been classified.  It is shown in \cite[Theorem A]{BobinskiGrzegorzGeiss04} that every such algebra $\Lambda$ is either derived equivalent to the category of representations of a Dynkin quiver, or derived equivalent to $\Lambda(r,n,m) = kQ(r,n,m) / I(r,n,m)$ where $kQ(r,n,m)$ is the path algebra over the quiver $Q(r,n,m)$ and $I(r,n,m)$ is a suitably chosen ideal of $kQ(r,n,m)$ (see Figure \ref{fig:DerivedDiscrete}).  Here, $m \geq 0$ and $n \geq r \geq 1$.  The algebra $\Lambda(r,n,m)$ has finite global dimension if and only if $n > r$.

\begin{figure}
	\centering
		\includegraphics[width=0.80\textwidth]{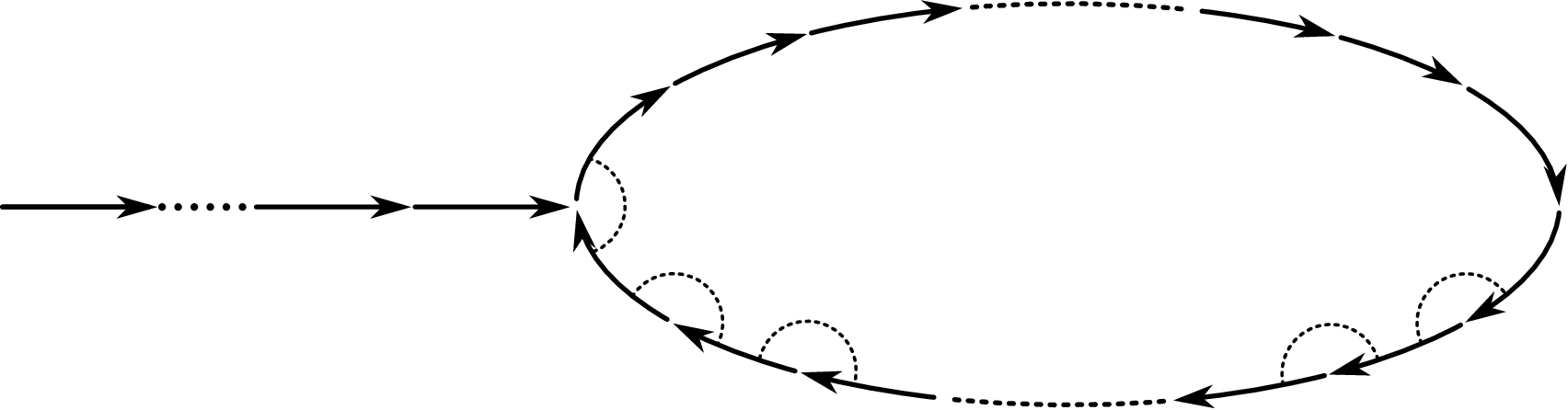}
	\caption{The quiver $Q(r,n,m)$.  Here, the tail has $m$ vertices, the oriented cycle has $n$ vertices.  The ideal $I(r,n,m)$ is given by $r$ consecutive quadratic relations in the cycle, the last one over the vertex to which the tail is attached.}
	\label{fig:DerivedDiscrete}
\end{figure}

The following result is a corollary of the description of the bounded $t$-structures on $\Db(\mod \Lambda)$ given in \cite{BroomheadPauksztelloPloog13} and Proposition \ref{proposition:FinitelyGeneratedAisle}.

\begin{corollary}\label{corollary:DerivedDiscrete}
Let $\Lambda$ be a finite-dimensional algebra of finite global dimension over an algebraically closed field.  Assume that $\Lambda$ is derived discrete.  A $t$-structure $(\UU, \VV)$ on $\Db(\mod \Lambda)$ induces a triangle equivalence $\Db \HH(\UU) \to \Db \mod \Lambda$ if and only if $(\UU, \VV)$ is bounded and $\bS \UU \subseteq \UU$.
\end{corollary}

\begin{proof}
Since $\Lambda$ has finite global dimension, we know that $\Db (\mod \Lambda)$ has a Serre functor.  It is shown in \cite{BroomheadPauksztelloPloog13} that all $t$-structures on $\Db(\mod \Lambda)$ are finitely generated; Proposition \ref{proposition:FinitelyGeneratedAisle} then yields the required result.
\end{proof}

\section{Projective objects in the heart}
Let $\AA$ be an abelian hereditary category with Serre duality, and let $(\UU, \VV)$ be a bounded $t$-structure on $\Db \AA$.  In general, a $\UU$-projective object does not lie in the heart $\HH(\UU)$.  The aim of this section is proving Proposition \ref{proposition:SimpleTop}, where it is shown, under the additional condition $\UU \subseteq \bS \UU$, that $\UU$-projective objects are projective objects in the heart $\HH(\UU)$ and that any indecomposable $\UU$-projective object has a simple top.  These simple tops will play a further role in \S\ref{section:Reduction} where we will consider their perpendicular subcategories.

We start by recalling some definitions.  Let $\BB$ be any abelian category and let $f: A \to C$ and $g: B \to C$ be any two morphisms in $\BB$.  We write $f \sim_C g$ if there are morphisms $h_1: A \to B$ and $h_2: B \to A$ such that $f = g \circ h_1$ and $g = f \circ h_2$.  This defines an equivalence relation on morphisms ending in $C$.  A morphism $f: A \to C$ is called \emph{right minimal} if, for every $h \in \End A$, $f = f \circ h$ implies that $h$ is an automorphism.

The following proposition is a straightforward adaptation of \cite[Proposition I.2.1]{ARS}.

\begin{proposition}\label{proposition:Minimal}
Let $\BB$ be an abelian category and let $f: A \to C$ be any morphism.  If $\dim_k \End A < \infty$, then here is a right minimal morphism $g: B \to C$ such that $f \sim_C g$.
\end{proposition}

\begin{proof}
Let $g: B \to C$ be a morphism such that $f \sim_C g$ and assume that $g$ has been chosen so that $\dim_k \End B < \infty$ is minimal with this property.  This can be done since $\dim_k \End A < \infty$.  We will show that $g: B \to C$ is right minimal.

For any $h \in \End B$ such that $g = g \circ h$, we get a commutative diagram:
\newdir{ (}{{}*!/-5pt/@^{(}}
$$\xymatrix{
B \ar@{->>}[r]^-{\pi} \ar[d]^{g} & \im h \ar@{{ (}->}[r]^-{i} \ar[d]^{g|_{\im h}} & B \ar[d]^g \\
C \ar@{=}[r] & C \ar@{=}[r]  & C
}$$
This gives a linear transformation $\End (\im h) \to \End B: \varphi \mapsto i \circ \varphi \circ \pi.$  Since $i$ and $\pi$ are a monomorphism and an epimorphism, respectively, the map $\End (\im h) \to \End B$ is an injection.  Using the minimality of $\dim_k \End B$, we find that $\End (\im h) \to \End B$ is an isomorphism, and thus there is a $\varphi \in \End(\im h)$ such that $i \circ \varphi \circ \pi = 1_B$.  This shows that $\pi$ and $i$ are isomorphisms and hence so is $i \circ \pi = h$.  We conclude that $g: B \to C$ is right minimal.
\end{proof}

We will use the following characterization of right minimal morphisms.

\begin{proposition}\label{proposition:EquivalentMinimal}
Let $\BB$ be an abelian category.  The following are equivalent for a morphism $f: A \to C$ where $\dim_k \End A < \infty$:
\begin{enumerate}
\item the morphism $f: A \to C$ is right minimal, and
\item for any nonzero direct summand $A'$ of $A$, the restriction $f|_{A'}: A' \to C$ of $f$ to $A'$ is nonzero.
\end{enumerate}
\end{proposition}

\begin{proof}
This is exactly \cite[Corollary I.2.3]{ARS} where \cite[Proposition I.2.1]{ARS} is replaced by Proposition \ref{proposition:Minimal}.
\end{proof}

\begin{remark}
Proposition \ref{proposition:EquivalentMinimal} holds under the weaker assumption that $\dim_k \End_{\BB / C}(f) < \infty$ instead of $\dim_k \End_\BB A < \infty$; the proof is identical.  Here, $\BB / C$ is the slice category of $\BB$ over $C$.
\end{remark}

A morphism $f: A \to C$ is called \emph{right almost split} if $f$ is not a split epimorphism and any morphism $B \to C$ either factors through $f: A \to C$ or is a split epimorphism.  A morphism which is both right minimal and right almost split is called \emph{minimal right almost split}.

We now come to the main result of this section.

\begin{proposition}\label{proposition:SimpleTop}
Let $\AA$ be an abelian category with Serre duality, and let $(\UU, \VV)$ be a bounded $t$-structure in $\Db \AA$.  Assume furthermore that $\bS \UU \subseteq \UU$.  Let $E \in \UU$ be an indecomposable $\UU$-projective, and let $\tau E \to M \to E \to \bS E$ be the Auslander-Reiten triangle built on $E$.  We have the following:
\begin{enumerate}
\item\label{Number1} $E, \bS E \in \HH(\UU)$,
\item\label{Number2} for all $n \not= 0$, we have $\Ext^n_{\Db \AA} (E, \HH(\UU)) = 0$ and $\Ext^n_{\HH(\UU)} (E, \HH(\UU)) = 0$,
\item\label{Number3} for all $n \not= 0$, we have $\Ext^n_{\Db \AA} (\HH(\UU), \bS E) = 0$ and $\Ext^n_{\HH(\UU)} (\HH(\UU), \bS E) = 0$,
\item\label{Number4} the composition $M_\UU \to M \to E$ is minimal right almost split in $\HH(\UU)$,
\item\label{Number5} the object $S_E \cong E / M_\UU$ is simple in $\HH(\UU)$,
\item\label{Number6} for any $X \in \HH(\UU)$, we have that $\Hom(E,X) = 0$ implies $\Hom(X,S_E) = 0$,
\item\label{Number7} for any $X \in \HH(\UU)$, we have that $\Hom(X,\bS E) = 0$ implies $\Hom(S_E,X) = 0$,
\item\label{Number8} if $\AA$ is hereditary, then $\Hom_{\Db \AA}(S_E,S_E[n]) = 0$ for all $n \not= 0$.
\end{enumerate}
\end{proposition}

\begin{proof}
\begin{enumerate}
  \item We first show that $E \in \HH(\UU)$.  We know that $E \in \UU$ so that we only need to show that $E \in \UU^{\perp_{-1}}$.  We have
$$\Hom(\UU,E[-1]) \cong \Hom(E, \bS \UU[1])^* = 0,$$
where the first isomorphism is by Serre duality, and the last equality follows from $\bS \UU \subseteq \UU$ and $\Hom(E,\UU[1]) = 0$.

  Next, we will show that $\bS E \in \HH(\UU)$.  Since $\bS \UU \subseteq \UU$, we know that $\bS E \in \UU$.  To show that $\bS E \in \HH(\UU)$, we need to show that $\Hom(X[1],\bS E) = 0$, for all $X \in \UU$.  This follows from Serre duality and $\Hom(E, X[1]) = 0$.

  \item Since $E$ is $\UU$-projective and $\HH(\UU) \subset \UU$, we have $\Ext^n_{\Db \AA} (E, \HH(\UU)) = 0$, for all $n > 0$.  It follows from the definition of $\HH(\UU)$ that $\Ext^n_{\Db \AA} (E, \HH(\UU)) = 0$ for $n < 0$.
  
  Furthermore, by Remark \ref{remark:ExtInHeart} we have $\Ext^1_{\Db \AA}(E,\HH(\UU)) \cong \Ext^1_{\HH(\UU)}(E,\HH(\UU))$.  Since the left-hand side is zero, and we may conclude that $E \in \HH(\UU)$ is projective in $\HH(\UU)$.  Hence $\Ext^n_{\HH(\UU)}(E,\HH(\UU)) = 0$ for all $n \not= 0$.

\item Using Serre duality, we find
\begin{eqnarray*}
\Ext^n_{\Db \AA} (\HH(\UU), \bS E) &=& \Hom_{\Db \AA} (\HH(\UU), \bS E[n]) \\
&\cong& \Hom_{\Db \AA} (E, \HH(\UU)[-n])^* \\
&\cong& \Ext^{-n}_{\Db \AA} (E, \HH(\UU))^* \\
\end{eqnarray*}
so that $\Ext^n_{\Db \AA} (\HH(\UU), \bS E) = 0$ for $n \not= 0$ by (\ref{Number2}).  Following Remark \ref{remark:ExtInHeart}, we have $\Ext^1_{\Db \AA}(\HH(\UU), \bS E) \cong \Ext^1_{\HH(\UU)}(\HH(\UU), \bS E)$, so that $\bS E \in \HH(\UU)$ is injective in $\HH(\UU)$ and consequently $\Ext^n_{\HH(\UU)} (\HH(\UU), \bS E) = 0$ for all $n \not= 0$.

\item First, we show that $M_\UU \in \HH(\UU)$.  We only need to check that $\Hom(X[1],M_\UU) = 0$, for all $X \in \UU$.  Since $X[1] \in \UU$, we find that $\Hom(X[1],M_\UU) \cong \Hom(X[1],M)$. Applying $\Hom(X[1],-)$ to the triangle defining $M$ shows it suffices to prove that $\Hom(X[1],E) = 0 = \Hom(X[1], \tau E)$.  The first equality follows from $E \in \HH(\UU)$, and the second equality follows from $\Hom(X[1], \tau E) \cong \Hom(X[2], \bS E) = 0$ together with $\bS E \in \HH(\UU)$.

We will use Proposition \ref{proposition:EquivalentMinimal} to show that $M_\UU \to E$ is right minimal.  Let $X \in \HH(\UU)$ be a direct summand of $M_\UU$ such that the composition $X \to M_\UU \to E$ is zero.  Using the triangle defining $M$, we find that the composition $X \to M_\UU \to M$ factors as $X \to \tau E \to M$.  It now follows from $\Hom(X, \tau E) \cong \Hom(X[1], \bS E) = 0$ (recall that $\bS E \in \HH(\UU)$) that $X \to M_\UU \to M$ is zero, and hence the embedding $X \to M_\UU$ is zero.  We conclude that $X \cong 0$, which shows that $M_\UU \to E$ is indeed right minimal.

To show that $M_\UU \to E$ is right almost split, we first show it is nonsplit.  If the map $M_\UU \to E$ were a split epimorphism, then the map $M \to E$ would be split as well, contradicting that $\tau E \to M \to E \to \bS E$ is an Auslander-Reiten triangle.  Next, let $X \to E$ be any nonsplit morphism in $\HH(\UU)$.  Using the properties of Auslander-Reiten triangles, we find that this morphism factors as
$$X \to M_\UU \to M \to E,$$
and hence $M_\UU \to E$ is right almost split.

\item The proof that $S_E$ is simple is standard.  We give the proof for the convenience of the reader.  Let $T$ be a quotient object of $S_E$.  We find a commutative diagram
$$\xymatrix{0 \ar[r] & M_\UU \ar@{-->}[d]\ar[r] & E \ar[r] \ar@{=}[d] & S_E \ar[r] \ar[d] & 0  \\
0 \ar[r] & K \ar[r] & E \ar[r] & T \ar[r] & 0}$$
It was shown in (\ref{Number4}) that the map $M_\UU \to E$ is right almost split, thus either $K \to E$ is a split epimorphism or $K \to E$ factors as $K \to M_\UU \to E$.  In the former case we find that $K \to E$ is an isomorphism (and hence $T \cong 0$), and in the latter case we can use that $M_\UU \to E$ is right minimal to find that the composition $M_\UU \to K \to M_\UU$ is an automorphism of $M_\UU$, implying that $T \cong S_E$.  This shows that $S_E$ is simple.

\item This property follows from the lifting property for projectives.
\item This property follows from the lifting property for injectives.
\item Assume now that $\AA$ is hereditary.  Since $S_E$ is simple in $\HH(\UU)$ we know that it is indecomposable, and since $\AA$ is hereditary there is an $i \in \bZ$ such that $S_E \in \AA[i]$.  Again, using that $\AA$ is hereditary, we know that $\Hom_{\Db \AA}(S_E,S_E[n]) = 0$, for all $n \not\in \{0,1\}$.  We thus only need to show that $\Hom_{\Db \AA}(S_E,S_E[1]) = 0$.  By Remark \ref{remark:ExtInHeart}, this is equivalent to showing that $\Ext^1_{\HH(\UU)}(S_E,S_E) = 0$.

Consider the short exact sequence $0 \to M_\UU \to E \to S_E \to 0$ in $\HH(\UU)$.  Applying the functor $\Hom_{\HH(\UU)}(-,S_E)$ and using (\ref{Number2}) to see that $\Ext^1_{\HH(\UU)}(E,S_E) = 0$, we find the exact sequence
$$0 \to \Hom(S_E,S_E) \to \Hom(E,S_E) \to \Hom(M_\UU, S_E) \to \Ext^1(S_E,S_E) \to 0.$$
Seeking a contradiction, assume that $\Ext^1_{\HH(\UU)}(S_E, S_E) \not= 0$.  This implies that $\Hom(M_\UU, S_E) \not= 0$ and by (\ref{Number6}) above thus also that $\Hom(E, M_\UU) \not= 0$.

However, recall that $E$ is $\UU$-projective and thus $\Ext^1(E,E) = 0$, so that Proposition \ref{proposition:HappelRingel} shows that every nonzero element in $\Hom(E,E)$ is invertible.  If the composition $E \to M_\UU \to E$ were nonzero, then it would be invertible.  In particular, the map $M_\UU \to E$ would be a split epimorphism, contradicting (\ref{Number4}) above.  We conclude that the composition $E \to M_\UU \to E$ is zero.  However, this contradicts that $M_\UU \to E$ is a monomorphism together with $\Hom(E, M_\UU) \not= 0$.
\end{enumerate}
\end{proof}

\section{A criterion for derived equivalence}\label{section:Criterion}

In this section, let $\Lambda$ be a finite-dimensional algebra of finite global dimension, and write $\AA$ for the category $\mod \Lambda$ of finite-dimensional right $\Lambda$-modules.  We will consider an aisle $\UU \subseteq \Db \AA$ and ``approximate'' it by a finitely generated aisle $\YY$ (this is made precise in Proposition \ref{proposition:Criterion1} below).  The positive integer $i$ measures how far $\YY$ lies from $\UU$.

Following Proposition \ref{proposition:Weight}, the aisle $\YY$ fits into a weight structure $(\XX, \YY)$ on $\Db \AA$, and we can use this weight structure to factor morphisms of the form $A \to B[n]$ where $A, B \in \HH(\UU)$.

When $\AA$ is hereditary, we can use the description of $t$-structures from \cite{StanleyvanRoomalen12} (see \S\ref{subsection:Classification}) to find such a finitely generated aisle $\YY$ (this will be done in Proposition \ref{proposition:Criterion} below).  This leads to Lemma \ref{lemma:OnlyExt2} where we show that whether $\UU \subseteq \Db \AA$ induces a derived equivalence is only ``controlled'' by morphisms of the form $A \to B[2]$ where $A, B \in \HH(\UU)$.

\begin{proposition}\label{proposition:Criterion1}
Let $\Lambda$ be a finite-dimensional algebra of finite global dimension, and write $\AA$ for $\mod \Lambda$.  Let $(\UU, \VV)$ be a bounded $t$-structure on $\Db \AA$.  If there is a finitely generated and bounded aisle $\YY \subseteq \Db \AA$ and an $i \geq 0$ such that 
\begin{enumerate}
\item $\bS \YY[i] \subseteq \UU \subseteq \YY$, and
\item $\YY[1] \subseteq \UU$,
\end{enumerate}
then every morphism in $\Hom_{\Db \AA}(A,B[n])$ (with $A,B \in \HH(\UU)$ and $n \geq i+2$) factors through an object $X \in \UU[n-i-1] \cap \VV[n-1]$.
\end{proposition}

\begin{proof}
Let $A,B \in \HH(\UU)$ and assume that $\Hom_{\Db \AA} (A,B[n]) \not= 0$ for some $n \geq i+2$.  Since $\YY$ is finitely generated, so is $\bS \YY[n]$.  By Proposition \ref{proposition:Weight}, there is a bounded weight structure $(\bS \XX[n], \bS \YY[n])$ on $\Db \AA$, and thus there is a triangle
$$X \to B[n] \to Y \to X[1]$$
where $B[n] \to Y$ is a left $\bS \YY[n]$-approximation (here, $X \in \bS \XX[n-1]$).  Using this triangle, we will show that $X \in \bS \XX[n-1]$ is the object in the statement of the lemma.

We start by showing that the morphism $A \to B[n]$ factors through $X$.  Since $A \in \HH(\UU)$, we know that $\Hom(A,\bS \UU [1]) \cong \Hom(\UU[1],A)^* = 0$.  Moreover, we have assumed that $\YY[1] \subseteq \UU$ and $n \geq 2$, so that $\Hom(A, \bS \YY [n]) = 0$.  Since $Y \in \bS \YY[n]$, this shows that the composition $A \to B[n] \to Y$ is zero and thus, by using the triangle above, the map $A \to B[n]$ factors as $A \to X \to B[n]$.

Next, we need to show that $X \in \UU[n-i-1] \cap \VV[n-1]$.  Note that
\begin{align*}
Y[-1] &\in \bS \YY[n-1] \subseteq \UU[n-i-1], \\
B[n] &\in \UU[n] \subseteq \UU[n-i-1].
\end{align*}
This implies that $X \in \UU[n-i-1]$.  Since $\Hom(\bS \XX[n-1], \bS \YY[n]) = 0$ and $X \in \bS \XX[n-1]$, we know that $\Hom(X,\bS \YY[n]) = 0$ and thus also $\Hom(\YY[n],X) = 0$.  Since $\UU[n] \subseteq \YY[n]$, this implies that $\Hom(\UU[n],X) = 0$ and we conclude that $X \in \VV[n-1]$.
\end{proof}

\begin{remark}\label{remark:YYbounded}
The aisle $\YY$ from Proposition \ref{proposition:Criterion1} is bounded since $\UU \subseteq \YY$ and $\YY[1] \subseteq \UU$.
\end{remark}

\begin{remark}\label{remark:YYclosed}
If one can choose $\YY$ as in Proposition \ref{proposition:Criterion1} such that $\bS \YY \subseteq \UU \subseteq \YY$ (thus $i=0$), then $\bS \UU \subseteq \UU$ and $\bS \YY \subseteq \YY$.
\end{remark}

\begin{corollary}\label{corollary:Criterion}
If the conditions of Proposition \ref{proposition:Criterion1} are met for $i = 0$, then:
\begin{enumerate}
  \item the aisle $\UU \subseteq \Db \AA$ induces a triangle equivalence $\Db \HH(\UU) \to \Db \AA$,
  \item the aisle $\YY \subseteq \Db \AA$ induces a triangle equivalence $\Db \HH(\YY) \to \Db \AA$, and
  \item $H_\YY^0 (\UU) \subseteq \HH(\YY)$ is a tilting torsion class.
\end{enumerate}
\end{corollary}

\begin{proof}
\begin{enumerate}
\item
It follows from Proposition \ref{proposition:Criterion1} that every morphism $A \to B[n]$ (with $A, B \in \HH(\UU)$, and $n \geq 2$) factors through some $Z \in \HH(\UU)[n-1]$.  The first statement then easily follows from Theorem \ref{theorem:BBD}.

\item
As in Remarks \ref{remark:YYbounded} and \ref{remark:YYclosed}, we see that $\YY$ is bounded and that $\bS \YY \subseteq \YY$ and thus the aisle $\YY \subseteq \Db \AA$ induces an equivalence $\Db \HH(\YY) \to \Db \AA$ by the first part of the proof.

\item 
We will first check that $H_\YY^0 (\UU) \subseteq \HH(\YY)$ is a torsion class (see also \cite[Lemma 3.5]{KingQiu15}).  Since $\HH(\YY) \cong \mod \Gamma$ for a finite-dimensional algebra $\Gamma$ (by Proposition \ref{proposition:FinitelyGeneratedAisle}), we know that $\HH(\YY)$ is noetherian, and hence it suffices to show that $H_\YY^0 (\UU)$ is closed under extensions and quotient objects.  Since $\UU$ is closed under extensions, so is $H_\YY^0(\UU)$ (see Remark \ref{remark:ExtInHeart}).  Thus, let $Z \in H_\YY^0(\UU)$ and consider the short exact sequence
$$0 \to K \to Z \to Q \to 0$$
in $\HH(\YY)$.  Since $K \in \HH(\YY) \subset \YY$, we know that $K[1] \in \YY[1] \subseteq \UU$.  We conclude that $Q[0] \in \UU$ and thus $Q \in H_\YY^0(\UU)$.  This shows that $H_\YY^0 (\UU) \subseteq \HH(\YY)$ is closed under quotient objects.  This establishes that $H_\YY^0 (\UU) \subseteq \HH(\YY)$ is a torsion class.

We will now show that $H_\YY^0 (\UU) \subseteq \HH(\YY)$ is a tilting torsion class.  Recall that there is a finite-dimensional algebra $\Gamma$ such that $\HH(\YY) \cong \mod \Gamma$.  Therefore, $\HH(\YY)$ has a projective generator $E$ so that $\bS E$ is then an injective cogenerator for $\HH(\YY)$.  Since $\bS \YY \subseteq \UU$, we know that $\bS E \in \UU$.  We conclude that $\bS E \in H_\YY^0 (\UU)$.
\end{enumerate}
\end{proof}

\begin{remark}
If the aisle $\UU \subseteq \Db \AA$ is bounded and finitely generated, then Corollary \ref{corollary:Criterion} implies that $\UU \subseteq \Db \AA$ induces a derived equivalence.  In this case, one can choose $\YY = \UU$.  This provides another proof of Proposition \ref{proposition:FinitelyGeneratedAisle}.
\end{remark}

The following example shows that, even when $\Lambda$ is hereditary, one cannot expect to find an aisle $\YY \in \Db \mod \Lambda$ as in Corollary \ref{corollary:Criterion}.

\begin{example}\label{example:NoYY}
We will assume that the reader is familiar with the representation theory of tame hereditary algebras.  Let $\Lambda$ be the path algebra $\bC Q$ where $Q$ is the quiver
$$\xymatrix@R=5pt{& b \ar[rd] \\
a\ar[ru] \ar[rr] && c}$$
The algebra $\Lambda$ is of tame representation type.  We will use the classification from \cite{StanleyvanRoomalen12} (see \S\ref{subsection:Classification}) to describe an aisle $\UU \subseteq \Db \mod \bC Q$.

Let $S_i$ be the simple corresponding to the vertex $i$, where $i$ is $a,b$, or $c$, and let $P_i$ be the projective cover of $S_i$.  Note that $\t^2 S_b \cong S_b$.

We will give an aisle $\UU$ in $\Db \mod \bC Q$ by given the associated refined $t$-sequence (see \S\ref{subsection:Classification}).  We will also write $\NN(n)$ for $H^{-n} (\UU)$.

We start by considering the following wide subcategories of $\mod \bC Q:$
$$\WW(n) = \begin{cases} 0 &\mbox{for $n < 0$} \\ 
  \wide(P_a \oplus P_c) & \mbox{for $n=0$,} \\
 \mod \bC Q & \mbox{for $n>0$.} \end{cases}$$

Note that $\WW(0)$ is equivalent to $\rep_\bC (\xymatrix{\cdot \ar@<-2pt>[r] \ar@<2pt>[r] & \cdot)}$.

For each $n \in \bZ$, we choose a tilting torsion class $t_\WW(n)$ in $\WW(n) \cap {}^\perp \WW(n-1)$.  Note that $\WW(n) \cap {}^\perp \WW(n-1)$ is nonzero only for $n=0$ and $n=1$.

The torsion classes in $\WW(0) \cap {}^\perp \WW(-1) = \WW(0)$ are given in \cite[Theorem 2]{AssemKerner96}.  Let $t_\WW(0)$ be the tilting torsion class given by all preinjectives in $\WW(0) \cap {}^\perp \WW(-1)$.

Since $\WW(1) \cap {}^\perp \WW(0) = \wide (\tau S_b) \cong \mod \bC$, there is only one tilting torsion class: we choose $t_\WW(1) = \WW(1) \cap {}^\perp \WW(0)$.

Proposition \ref{proposition:FinitelyGenerated} yields that $\tau S_b$ is $\NN(1)$-projective, and thus $S_b \not\in \NN(1)$.  Indeed, $\Ext^1(\t S_b, S_b) \cong \Hom(S_b, S_b) \not= 0$.

Assume now the existence of an aisle $\YY$ as in Corollary \ref{corollary:Criterion}.  In particular, $\UU \subseteq \YY$ and thus $\NN(0) \subseteq H^0(\YY)$.  There are two possibilities for $\wide(H^0(\YY))$.

The first possibility is $\wide(H^0(\YY)) = \mod \bC Q$.  In this case, it follows from \cite[Theorem 2]{AssemKerner96} that the finitely generated tilting torsion theories containing $\NN(0)$ must contain $S_b$ or $\tau S_b$.  However, since $\YY[1] \subseteq \UU$ and $\bS \YY \subseteq \UU$, this shows that $S_b \in \NN(1)$.  Contradiction.

The second possibility is $\wide(H^0(\YY)) = \WW(0)$.  Here, we can use \cite[Theorem 2]{AssemKerner96} to see that a finitely generated torsion class containing all preinjective objects contains all regular objects in $\WW(0)$ as well.  Hence, $H^0(\YY)$ contains the middle term $M$ of the Auslander-Reiten sequence
$$0 \to \tau S_b \to M \to S_b \to 0.$$

Using that $\YY[1] \subseteq \UU$, we see that $M \in \NN(1)$ and since $\NN(1) \subseteq \WW(1)$ is a torsion class, we have $S_b \in \NN(1)$.  Contradiction.

We conclude that there can be no finitely generated aisle $\YY$ satisfying the conditions of Corollary \ref{corollary:Criterion}.
\end{example}

\begin{proposition}\label{proposition:Criterion}
Let $\Lambda$ be a finite-dimensional hereditary algebra and let $\UU \subseteq \Db \mod \Lambda$ be a bounded aisle. If $\bS \UU \subseteq \UU$, then there is a finitely generated aisle $\YY \subseteq \Db \mod \Lambda$ satisfying the conditions from Proposition \ref{proposition:Criterion1} for $i=1$.
\end{proposition}

\begin{proof}
By the description of aisles in $\Db \mod \Lambda$ (see \S\ref{subsection:Classification}), we know that $\UU$ is given by a refined $t$-sequence $(\WW(-), t_\WW(-))$.  Let $s_\WW(n) = \WW(n) \cap {}^\perp \WW(n-1)$, thus whereas $t_\WW(n)$ is any tilting torsion class in $\WW(n) \cap {}^\perp \WW(n-1)$, we have chosen $s_\WW(n)$ to be the maximal one.  Moreover, due to Remark \ref{remark:FormOfPerpendiculars}, we know that $s_\WW(n)$ is finitely generated.

Let $\YY \subseteq \Db \mod \Lambda$ be the aisle associated to $(\WW(-), s_\WW(-))$.  We know that $\YY$ is finitely generated and bounded by Corollary \ref{corollary:FinitelyGeneratedCriterion}.

Recall from Remark \ref{remark:Construction} that $\YY$ is the smallest preaisle in $\Db \mod \Lambda$ which contains $s_\WW(n)[n]$ for all $n \in \bZ$.  In particular, $\UU \subseteq \YY$.  Furthermore, it follows from Proposition \ref{proposition:AislesInclusion} that $\YY[1] \subseteq \UU$, and since $\bS \UU \subseteq \UU$, we have $\bS \YY[1] \subseteq \UU$.
\end{proof}

Following Theorem \ref{theorem:BBD}, we can study derived equivalences of $\mod \Lambda$ by studying morphisms in $\Db \mod \Lambda$ of type $A \to B[n]$ (for $A,B \in \HH(\UU)$) for all $n \geq 2$.  The next lemma explains why we may, in our case, reduce to only considering the case $n=2$.

\begin{lemma}\label{lemma:OnlyExt2}
Let $\Lambda$ be a finite-dimensional hereditary algebra and let $\UU \subseteq \Db \mod \Lambda$ be a bounded aisle satisfying $\bS \UU \subseteq \UU$.  If every morphism $A \to B[2]$ in $\Db \mod \Lambda$ (for $A,B \in \HH(\UU)$) factors as $A \to Z[1] \to B[2]$ (for some $Z \in \HH(\UU)$), then the aisle $\UU \subseteq \Db \mod \Lambda$ induces a derived equivalence.
\end{lemma}

\begin{proof}
We will show that, for all objects $A,B \in \HH(\UU)$ and all integers $n \geq 2$, there is a monomorphism $B \to C$ in $\HH(\UU)$ such that $A \to B[n] \to C[n]$ is zero.  The required property then follows from Theorem \ref{theorem:BBD}.

Seeking a contradiction, assume that we have chosen $A,B \in \HH(\UU)$ and $n \geq 2$ such that there exists no such monomorphism $B \to C$.  Furthermore, assume that (over all $A,B \in \HH(\UU)$), we have chosen $n$ to be minimal with this property.  The conditions in the statement of the lemma imply that $n \geq 3$ (see Lemma \ref{lemma:DerivedEquivalence}).

It follows from Propositions \ref{proposition:Criterion1} and \ref{proposition:Criterion} that there is an object $X \in \UU[n-2] \cap \VV[n-1]$ and a factorization $A \to X \to B[n]$.  We have a triangle
$$(H_\UU^{1-n}X)[n-1] \to X \to (H_\UU^{2-n}X)[n-2] \to (H_\UU^{1-n}X)[n].$$
The composition $(H_\UU^{1-n}X)[n-1] \to X \to B[n]$ induces a morphism of triangles:
$$\xymatrix{
(H_\UU^{1-n}X)[n-1] \ar[r] \ar@{=}[d] & X \ar[r] \ar[d] & (H_\UU^{2-n}X)[n-2] \ar[r] \ar@{-->}[d]& (H_\UU^{1-n}X)[n] \ar@{=}[d] \\
(H_\UU^{1-n}X)[n-1] \ar[r] & B[n] \ar[r] & M[n] \ar[r] & (H_\UU^{1-n}X)[n]
}$$
We note that $M \in \HH(\UU)$.  By Remark \ref{remark:ExtInHeart}, the bottom triangle gives a short exact sequence
$$0 \to B \to M \to H_\UU^{1-n}X \to 0$$
in $\HH(\UU)$.  We see that the composition $X \to B[n] \to M[n]$ factors as $X \to (H_\UU^{2-n}X)[n-2] \to M[n]$.  By the assumptions of the lemma, we know that there is a $Z \in \HH(\UU)$ such that $H_\UU^{2-n}X \to M[2]$ factors through $Z[1]$.  There is thus a short exact sequence
$$0 \to M \to C \to Z \to 0$$
in $\HH(\UU)$.  We claim that the composition $B \to M \to C$ is the required monomorphism in $\HH(\UU)$.

To verify this claim, it suffices to check that the composition $X \to B[n] \to M[n] \to C[n]$ is zero or, since $X \to B[n] \to M[n]$ factors as $X \to (H_\UU^{2-n}X)[n-2] \to M[n]$, that $(H_\UU^{2-n}X)[n-2] \to M[n] \to C[n]$ is zero.  The last statement follows easily since $(H_\UU^{2-n}X)[n-2] \to M[n]$ factors through $Z[n-1] \to M[n]$.
\end{proof}

We will use the following equivalent formulation for $\Ext_{\Db \AA}^2(A,B)$ to be the Yoneda composition of two short exact sequences in $\HH(\UU)$.

\begin{lemma}\label{lemma:Ext2Splittings}
Let $\AA$ be an abelian category, and let $\UU \subseteq \Db \AA$ be an aisle.  For all $A,B \in \HH(\UU)$ and all morphisms $f:A \to B[2]$ in $\Db \AA$, the following are equivalent:
\begin{enumerate}
\item there is a $Z \in \HH(\UU)$ such that $f$ factors as $A \to Z[1] \to B[2]$,
\item there is an epimorphism $C \to A$ in $\HH(\UU)$ such that the composition $C \to A \to B[2]$ is zero, and
\item there is an monomorphism $B \to D$ in $\HH(\UU)$ such that the composition $C \to B[2] \to D[2]$ is zero.
\end{enumerate}
\end{lemma}

\begin{proof}
Directly from Lemma \ref{lemma:DerivedEquivalence}.
\end{proof}
\section{Aisles with no nonzero Ext-projectives}\label{section:NoExtProjectives}

Let $\Lambda$ be a finite-dimensional hereditary algebra and let $\AA$ be the category $\mod \Lambda$ of finite-dimensional right $\Lambda$-modules.  In this section, we will use the results of \S\ref{section:Criterion} to show that our main theorem holds under the additional assumption that the aisle $\UU \subseteq \Db \AA$ has no nonzero $\UU$-projectives.

We will use the description of the aisles in $\Db \AA$, given in \S\ref{subsection:Classification}.  Thus let $\UU \subseteq \Db \AA$ be an aisle with no nonzero $\UU$-projectives, and let $(\WW(-),t_\WW(-))$ be the associated refined $t$-sequence.  We will write $\NN(n)$ for $H^{-n}(\UU)$.  Recall that $\WW(n)$ is $\wide \NN(n)$ and that $t_\WW(n) = \NN(n) \cap {}^\perp \WW(n-1)$.

Since we assume that there are no nonzero $\UU$-projective objects, we know from Proposition \ref{proposition:FinitelyGenerated} that $t_\WW(n)$ has no nonzero $t_\WW(n)$-projective objects, for all $n \in \bZ$.  However, $\NN(n)$ may have nonzero $\NN(n)$-projective objects.  The following lemma describes the possible $\NN(n)$-projective objects.

\begin{lemma}\label{lemma:WhatAreProjectives}
If $t_\WW(n)$ has no nonzero $t_\WW(n)$-projective objects, then every $\NN(n)$-projective object lies in $\WW(n-1)$ and is $\WW(n-1)$-projective.
\end{lemma}

\begin{proof}
Let $E$ be an $\NN(n)$-projective object.  It follows from Proposition \ref{proposition:RecoverFromSequence} that there is a triangle
$$A[0] \to E[0] \to B \to A[1]$$
where $B \cong E^{\Db \WW(n-1)}$ and $A \in t_\WW(n)$.

We claim that $B$ is concentrated in degree zero and that $H^0 B$ is $\WW(n-1)$-projective.  To verify this, we will show that $\Hom(B,B'[i])=0$ for all $B' \in \WW(n-1)$ and all $i \not= 0$.  Thus, let $B' \in \WW(n-1)$ and apply the functor $\Hom_{\Db \AA}(-,B'[i])$ to the above triangle to obtain the exact sequence
$$\Hom(A[1], B'[i]) \to \Hom(B, B'[i]) \to \Hom(E[0], B'[i]).$$
If $i < 0$, we see that $\Hom(B, B'[i]) = 0$.  For $i \geq 1$, we see that $\Hom(E[0],B'[i]) = 0$ (since $E$ is $\NN(n)$-projective and $B' \in \WW(n-1) \subseteq \NN(n)$) and that $\Hom(A[1],B'[i]) = 0$ (since $A \in t_\WW(n) \subseteq {}^{\perp}\WW(n-1)$).  Again, we may conclude that $\Hom(B,B'[i]) = 0$.  This implies that $B$ is concentrated in degree zero and that $H^0 B$ is $\WW(n-1)$-projective.  In particular, the triangle corresponds to a short exact sequence
$$0 \to A \to E \to H^0 B \to 0$$
in $\AA$.

Consider an object $A' \in t_\WW(n) \subseteq \NN(n)$.  Since $\Ext^1_\AA(-,A')$ is right exact (because $\AA$ is hereditary) and $\Ext^1(E,A') = 0$, we find that $\Ext^1(A,A') = 0$.  This shows that $A$ is a $t_\WW(n)$-projective, and hence we find that $A$ is zero.

We conclude that $E \cong H^0 B$, and we already established that $H^0 B$ is $\WW(n-1)$-projective.
\end{proof}

\begin{lemma}\label{lemma:ClosedUnderTau}
Let $\UU \subseteq \Db \AA$ be an aisle such that $\bS \UU \subseteq \UU$.  Let $(\WW(-),t_\WW(-))$ be the associated refined $t$-sequence.  If $t_\WW(n)$ has no nonzero $t_\WW(n)$-projective objects, then for any nonprojective indecomposable $B \in \WW(n-1)$ we have $\t B \in \NN(n)$.
\end{lemma}

\begin{proof}
We first establish that $\t B \in \WW(n)$.  Let $0 \to B \to I \to J \to 0$ be a minimal injective resolution of $B$ in $\WW(n-1)$.  Recall from \S\ref{subsection:Classification} that $\NN(n-1)$ is a tilting torsion class in $\WW(n-1)$ so that we know that $I,J \in \NN(n-1)$.

Next, we claim that $I, J$ have no nonzero direct summands which are $\AA$-projective.  Since the injective resolution is minimal, we know that $J$ has no $\WW(n-1)$-projective direct summands, and thus in particular, $J$ has no $\AA$-projective direct summands.  Since $\WW(n-1)$ is hereditary and $B$ is an indecomposable nonprojective object, it follows from the minimality of the injective resolution know that $I$ does not contain (nonzero) projective direct summands in $\WW(n-1)$ and hence $I$ has no nonzero projective direct summands in $\AA$.

We may now apply $\t$ to obtain a short exact sequence $0 \to \t B \to \t I \to \t J \to 0$ in $\AA$.  It follows from $\bS \UU \subseteq \UU$ that $\tau \NN(n-1) \subseteq \NN(n)$, so that $\t I, \t J \in \NN(n) \subseteq \AA$ and thus $\t B \in \WW(n)$.

Using that $\NN(n)$ is a tilting torsion class in $\WW(n)$, we obtain the short exact sequence
$$0 \to T \to \t B \to F \to 0$$
where $T \in \NN(n)$ is torsion and $F \cong \t B / T$ is torsionfree.  Since $\NN(n)$ is tilting in $\WW(n)$, we know that all $\WW(n)$-injective objects lie in $\NN(n)$ and hence $F$ does not contain (nonzero) injective direct summands in $\WW(n)$.  We want to show that $\t B \in \NN(n)$ by showing that $F = 0$.

We know from Proposition \ref{proposition:Wide} that $\WW(n) \cong \mod \Gamma$, for a finite-dimensional hereditary algebra $\Gamma$, and hence $\WW(n)$ has an Auslander-Reiten translation $\t_n$.  Note that $B,\t B \in \WW(n)$ so that $\t_n B \cong \t B$ and thus $\t_n^- \t B \cong B$.

Since $\t_n^-$ is right exact and $\t_n^- \t B \cong B$, there is an epimorphism $B \to \t_n^- F$.  It follows from $B \in \WW(n-1) \subseteq \NN(n)$ (this last inclusion was shown in \cite[Corollary 4.4]{StanleyvanRoomalen12}) that $B$ is torsion and we infer that $\t^- F$ is torsion.  Using that $F$ has no injective direct summands in $\WW(n)$ and hence $\tau_n \tau_n^- F \cong F$, we see that $\t_n^- F$ is zero if and only if $F$ is zero.  Thus to show that $\t B \in \NN(n)$, it suffices to show that $\t_n^- F$ is zero.

It follows from the Auslander-Reiten formula that $\Ext_{\WW(n)}^1(\t_n^- F, N) = 0$ for all $N \in \NN(n)$, hence $\t^- F$ is $\NN(n)$-projective.  Lemma \ref{lemma:WhatAreProjectives} implies that $\t_n^- F$ is a projective object in $\WW(n-1)$.  Since $B$ has no $\WW(n-1)$-projective direct summands, we can use that $\WW(n-1)$ is hereditary to infer that $\Hom(B, \t_n^- F) = 0$.  Since $\t_n^- F$ is a quotient object of $B$, we have established that $\t^- F$ is zero.  This finishes the proof.
\end{proof}

\begin{proposition}\label{proposition:NoProjectivesInAisle}
Let $\UU \subseteq \Db \AA$ be a bounded aisle without nonzero $\UU$-projective objects and such that $\bS \UU \subseteq \UU$.  There is a finitely generated aisle $\YY \subseteq \Db \AA$ satisfying the conditions in Proposition \ref{proposition:Criterion1} for $i=0$.
\end{proposition}

\begin{proof}
Let $(\WW(-), t_\WW(-))$ be the refined $t$-sequence associated with $\UU$.  We let $m_{\WW}(n)$ be the full subcategory of $\WW(n) \cap {}^\perp \WW(n-1)$ given by all objects without nonzero $(\WW(n) \cap {}^\perp \WW(n-1))$-projective direct summands, for all $n \in \bZ$.  Note that $m_{\WW}(n)$ is closed under quotient objects in $\WW(n) \cap {}^\perp \WW(n-1)$, and that $m_{\WW}(n)$ is closed under extensions since $\WW(n) \cap {}^\perp \WW(n-1)$ is hereditary (see Remark \ref{remark:FormOfPerpendiculars}).  Hence, $m_{\WW}(n)$ is a torsion class in $\WW(n) \cap {}^\perp \WW(n-1)$.  Since $t_\WW(n)$ does not contain any $\WW(n)$-projective objects, we know that $t_{\WW}(n) \subseteq m_{\WW}(n)$, and since $t_{\WW}(n)$ is tilting in $\WW(n) \cap {}^\perp \WW(n-1)$, we know that $m_{\WW}(n)$ is a tilting torsion class in $\WW(n) \cap {}^\perp \WW(n-1)$.

We will define $\YY \subseteq \Db \AA$ to be the aisle associated to the refined $t$-sequence $(\WW(-), m_\WW(-))$.  We want to use Corollary \ref{corollary:FinitelyGeneratedCriterion} to show that $\YY$ is indeed finitely generated and bounded.  Since $\UU$ is bounded, Corollary \ref{corollary:FinitelyGeneratedCriterion} implies that $\WW(n) = 0$ for $n \ll 0$ and that $\WW(n) = \AA$ for $n \gg 0$.  To see that $m_\WW(n)$ is finitely generated in $\WW(n) \cap {}^\perp \WW(n-1)$, recall from Remark \ref{remark:FormOfPerpendiculars} that $\WW(n) \cap {}^\perp \WW(n-1) \cong \mod \Gamma_n$, for a finite-dimensional hereditary algebra $\Gamma_n$; let $\tau_n$ be the Auslander-Reiten translate in $\WW(n) \cap {}^\perp \WW(n-1) \cong \mod \Gamma_n$.  In $\WW(n) \cap {}^\perp \WW(n-1)$, there are no indecomposable projective-injective objects.  Indeed, such an object would need to be contained in $t_\WW(n)$ and would hence be $t_\WW(n)$-projective.  The subcategory $m_\WW(n)$ of $\WW(n) \cap {}^\perp \WW(n-1)$ is generated by $\tau^-_n \Gamma_n$.  We can now apply Corollary \ref{corollary:FinitelyGeneratedCriterion} to see that $\YY$ is indeed finitely generated and bounded.

The inclusions $\YY[1] \subseteq \UU \subseteq \YY$ follows directly from Proposition \ref{proposition:AislesInclusion}.  We need to show that $\bS \YY \subseteq \UU$.  For this, we first to note that $m_\WW(n)$ does not contain nonzero $\AA$-projective objects.  Indeed, such nonzero $\AA$-projective objects would the be projective objects in $\WW(n) \cap {}^\perp \WW(n-1)$, and thus not contained in $m_\WW(n)$.  Hence, $\tau$ is defined on every object in $m_\WW(n)$.

It is straightforward to see that $\bS \YY \subseteq \UU$ if and only if $\tau m_\WW(n-1) \subseteq \NN(n)$, for all $n \in \bZ$.  We can conclude the proof by invoking Lemma \ref{lemma:ClosedUnderTau}.
\end{proof}

\begin{corollary}\label{corollary:NoProjectivesInAisle}
Let $\Lambda$ be a finite-dimensional hereditary algebra and let $(\UU, \VV)$ be a bounded $t$-structure on $\Db \mod \Lambda$ such that $\bS \UU \subseteq \UU$.  If $\UU$ has no nonzero $\UU$-projectives, then $(\UU, \VV)$ induces a triangle equivalence $\Db \HH(\UU)\stackrel{\sim}{\rightarrow}\Db \mod \Lambda$.
\end{corollary}

\begin{proof}
This follows from Corollary \ref{corollary:Criterion} and Proposition \ref{proposition:NoProjectivesInAisle}.
\end{proof}
\section{Reduction by a simple top}\label{section:Reduction}

Let $\AA$ be a hereditary category with Serre duality.  Let $\UU \subseteq \Db \AA$ be a bounded aisle, satisfying $\bS \UU \subseteq \UU$, and let $E \in \UU$ be an indecomposable $\UU$-projective.  It has been shown in Proposition \ref{proposition:SimpleTop} that $E$ is a projective object in the heart $\HH(\UU)$ and that $E$ has a simple top $S = S_E \in \HH(\UU)$.  By Schur's Lemma, $\End S$ is a skew field and the embedding $\thick S \to \Db \AA$ has a left and a right adjoint as described in \S\ref{subsection:Perpendicular}.

Recall from \S\ref{subsection:Perpendicular} that the embedding ${}^\perp S \to \Db \AA$ admits a left and a right adjoint, so that Proposition \ref{proposition:KellerSerre} yields that ${}^\perp S$ has a Serre functor $\bS' \cong T^*_S \circ \bS$.  To reduce notation, we write $\UU'$ for $\UU \cap {}^\perp S$.

The main result of this section is Proposition \ref{proposition:Reduction} below, where we reduce the problem of whether an aisle $\UU \subseteq \Db \AA$ induces a derived equivalence, to the (smaller and supposedly easier case) case of whether the corresponding aisle $\UU' \subseteq {}^\perp S$ induces a derived equivalence.

Our first step will be checking whether $\UU'$ is an aisle in ${}^\perp S$ which is closed under the Serre functor $\bS'$ of ${}^\perp S$.  This will be done in Proposition \ref{proposition:UUaisle}.

\begin{lemma}\label{lemma:UUaisle}
If $X \in \UU$, then $T^*_S(X) \in \UU'$.
\end{lemma}

\begin{proof}
We will show that $T^*_S(X) \in \UU$ by showing that $H_\UU^i(T_S^*(X)) = 0$ for all $i > 0$.

Consider the triangle
$$T^*_S(X) \to X \to \RHom(X,S)^* \stackrel{L}{\otimes}_{\End S} S \to T^*_S(X)[1]$$
and recall that $\RHom(X,S)^* \stackrel{L}{\otimes} S \cong \oplus_i \Hom(X,S[i])^* \otimes S[i]$.  Since $X \in \UU$ and $S \in \HH(\UU)$, we have $\Hom(X,S[i]) = 0$ for $i < 0$.  In particular, $H_{\UU}^i(\RHom(X,S)^* \stackrel{L}{\otimes} S) = 0$ for $i > 0$.

The long exact sequence obtained from the above triangle by applying $H_\UU$, shows that $H_\UU^i(T^*_S(X)) = 0$ for all $i > 1$.  To show that $T^*_S(X) \in \UU$, we need to show that $H_\UU^{1}(T^*_S(X)) = 0$.  This follows from the exactness of
$$H_\UU^0 (X) \to H_{\UU}^0 (\RHom(X,S)^* \stackrel{L}{\otimes} S) \to H_\UU^1(T^*_S(X)) \to H_\UU^1 (X),$$
together with $H_\UU^1 (X) = 0$ (since $X \in \UU$) and that $S$ is simple in $\HH(\UU)$ (so that the map $H_\UU^0 (X) \to H_{\UU}^0 (\RHom(X,S)^* \otimes S)$ is an epimorphism).
\end{proof}

\begin{proposition}\label{proposition:UUaisle}
With notation as above, $\UU'$ is an aisle in ${}^\perp S$ and $\bS' \UU' \subseteq \UU'$.
\end{proposition}

\begin{proof}
Note that $\UU'$ is indeed a preaisle in ${}^\perp S$.  It follows from Lemma \ref{lemma:UUaisle} that the functor $T^*_S \circ (-)_\UU$ takes images in $\UU'$.  It is straightforward to check that $T^*_S \circ (-)_\UU$ is right adjoint to the embedding $\UU' \to {}^\perp S$, and hence $\UU'$ is an aisle in ${}^\perp S$.

It follows from $\bS' \cong T^*_S \circ \bS$ and Lemma \ref{lemma:UUaisle}, together with $\bS \UU \subseteq \UU$, that $\bS' \UU' \subseteq \UU'$.
\end{proof}

We will write $\VV' = (\UU')^{\perp_{-1}}$ for the coaisle associated to the aisle $\UU'$ in ${}^\perp S$.  Note that $\VV'$ is not $\VV \cap {}^\perp S$.  The following lemma gives us control over the difference between $\VV'$ and $\VV \cap {}^\perp S$.

\begin{lemma}\label{lemma:WhenInVV}
For any $X \in {}^\perp S$, we have that $X \in \VV'$ if and only if $(X[-1])_\UU \in \thick(S)$.
\end{lemma}

\begin{proof}
First, assume that $(X[-1])_\UU \in \thick(S)$.  Let $Y \in \UU' \subseteq \UU$.  We know that
$$\Hom(Y,X[-1]) \cong \Hom(Y,(X[-1])_\UU) = 0,$$
which shows that $X \in \VV'$.

For the other direction, assume that $X \in \VV'$, or equivalently that $(X[-1])_{\UU'} = 0$.  By Proposition \ref{proposition:UUaisle}, we know that $(-)_{\UU'} \cong T^*_S \circ (-)_\UU$, so that $(X[-1])_{\UU}$ lies in the kernel of $T_S^*$.  This shows that $(X[-1])_\UU \in \thick S$ (see Remark \ref{remark:KernelOfTwist}).
\end{proof}

Since $\VV'$ is not a subcategory of $\VV$, we cannot expect $\HH(\UU')$ to be a subcategory of $\HH(\UU)$.  The following lemma relates the hearts $\HH(\UU)$ and $\HH(\UU')$.

\begin{lemma}\label{lemma:RelatesHearts}\begin{enumerate}
\item Let $A' \in \HH(\UU')$.  There is a triangle
$$\bigoplus_{i \geq 1} \Hom(S[i],A') \otimes S[i] \to A' \to H^0_\UU(A') \to \bigoplus_{i \geq 2} \Hom(S[i-1],A') \otimes S[i].$$
\item Let $A \in \HH(\UU)$.  We have $T^*_S(A) \in \HH(\UU')$.
\end{enumerate}
\end{lemma}

\begin{proof} \begin{enumerate}
\item First note that $\UU' \subseteq \UU$, so that $A' \in \UU$ and there is a triangle
$$A'_{\UU[1]} \to A' \to H^0_\UU(A') \to A'_{\UU[1]}[1].$$
We will proof the result by writing down a more explicit form of $A'_{\UU[1]}$.  Note that $A'_{\UU[1]} \cong (A'[-1])_{\UU}[1]$.  Since $A' \in \VV'$, Lemma \ref{lemma:WhenInVV} implies that $(A'[-1])_\UU \in \thick S$, thus $(A'[-1])_\UU$ has the following form:
$$(A'[-1])_\UU \cong \bigoplus_{i \in \bZ} V_i \otimes_{\End S} S[i],$$
where $V_i \in \mod (\End S)$.  Moreover, since $S \in \HH(\UU)$, we know that $S[i] \in \UU$ if and only if $i \geq 0$.  We may thus conclude that $V_i = 0$ when $i < 0$.  We determine $V_n$ (for $n \geq 0$) by considering
\begin{align*}
\Hom(S[n], A'[-1]) & \cong \Hom(S[n], (A'[-1])_\UU) \\
& \cong \Hom(S[n], \bigoplus_{i \in \bZ} V_i \otimes S[i]) \\
& \cong \bigoplus_{i \in \bZ} V_i \otimes \Hom(S[n], S[i]) \\
& \cong V_n,
\end{align*}
where we have used that $\Hom(S,S[i]) = 0$, for all $i \not= 0$.  We thus find
\begin{align*}
A'_{\UU[1]} & \cong (A'[-1])_{\UU}[1] \\
&\cong \bigoplus_{i \geq 0} \Hom(S[i], A'[-1]) \otimes S[i+1] \\
&\cong \bigoplus_{i \geq 0} \Hom(S[i+1], A') \otimes S[i+1] \\
&\cong \bigoplus_{i \geq 1} \Hom(S[i], A') \otimes S[i],
\end{align*}
which shows the required property.
\item By Lemma \ref{lemma:UUaisle}, we know that $T^*_S(A) \in \UU'$.  To show that $T^*_S(A) \in \HH(\UU')$, consider an object $X' \in \UU' \subseteq \UU$.  The functor $T^*_S: \Db \AA \to {}^\perp S$ is right adjoint to the embedding ${}^\perp S \to \Db \AA$, and thus $\Hom(X'[1], T^*_S(A)) \cong \Hom(X'[1], A)$.  We know that the latter is zero, and hence also the former.  This shows that $T^*_S(A) \in \HH(\UU')$.
\end{enumerate}
\end{proof}

\begin{lemma}\label{lemma:FurtherReduction}
Let $\AA$ be a hereditary category with Serre duality, and let $\UU \subseteq \Db \AA$ be a bounded aisle satisfying $\bS \UU \subseteq \UU$.  The following are equivalent:
\begin{enumerate}
  \item for all $A,B \in \HH(\UU)$ and all morphisms $f:A \to B[2]$ there is a $Z \in \HH(\UU)$ such that $f$ factors as $A \to Z[1] \to B[2]$,
  \item for all $A,B \in \HH(\UU)$ (satisfying additionally that $\Hom(A,S) = 0 = \Hom(S,B)$) and all morphisms $f:A \to B[2]$ there is a $Z \in \HH(\UU)$ such that $f$ factors as $A \to Z[1] \to B[2]$,
\end{enumerate}
\end{lemma}

\begin{proof}
We only need to show that the last statement implies the first.  Thus let $A,B \in \HH(\UU)$ and let $f:A \to B[2]$ be a morphism.  We will proceed in two steps.  In the first step, we will ``enlarge'' $A$ and reduce to where $\Hom(S,B) = 0$; in the second step, we will ``enlarge'' $B$ and reduce to where $\Hom(A,S) = 0$.

To ease notation, we will write $P_S$ and $I_S$ for the projective cover and injective envelope of $S$ in $\HH(\UU)$, respectively (see Proposition \ref{proposition:SimpleTop}, thus $P_S \cong E$ and $I_S \cong \bS E$).

For the first step, we follow Lemma \ref{lemma:Ext2Splittings} which states that the required factoring of $f$ would follow from the existence of an epimorphism $C \to A$ such that $C \to A \to B[2]$ is zero.  Consider the triangle
$$B \to \Hom(B,I_S)^* \otimes I_S \to C_B \to B[1]$$
in $\Db \AA$ built on the co-evaluation morphism $B \to \Hom(B,I_S)^* \otimes I_S$.  Applying the cohomological functor $H^0_\UU(-)$ gives the following exact sequence in $\HH(\UU)$:
$$0 \to K_B \to B \to \Hom(B,I_S)^* \otimes I_S \to Q_B \to 0.$$
We will also consider the triangle
$$K_B[1] \to C_B \to Q_B \to K_B[2]$$
in $\Db \AA$.  We will show that $\Hom(K_B, I_S) = 0$.  Using the lifting property of the injective $I_S$, any morphism $K_B \to I_S$ factors through the embedding $K_B \to B$.  Since the map $B \to \Hom(B,I_S)^* \otimes I_S$ is universal, the morphism $B \to I_S$ factors through $B \to \Hom(B,I_S)^* \otimes I_S$.  Thus the map $K_B \to I_S$ factors as
$$K_B \to B \to \Hom(B,I_S)^* \otimes I_S \to I_S,$$
which is zero.  By Proposition \ref{proposition:SimpleTop}(\ref{Number7}), we also know that $\Hom(S,K_B) = 0$.

Since $\Hom_{\Db \AA}(A,I_S[2]) = 0$ (by Proposition \ref{proposition:SimpleTop}(\ref{Number3})), we know that $f$ factors as $A \to C_B[1] \to B[2]$.  Consider the following morphism of triangles:
$$\xymatrix{
Q_B \ar[r] \ar@{=}[d] & M \ar[r] \ar[d] & A \ar[r] \ar[d] & Q_B[1] \ar@{=}[d] \\
Q_B \ar[r] & K_B[2] \ar[r] & C_B[1] \ar[r] & Q_B[1]
}$$
The topmost triangle corresponds to the short exact sequence $0 \to Q_B \to M \to A \to 0$ in $\HH(\UU)$ (see Remark \ref{remark:ExtInHeart}), so that the map $M \to A$ is an epimorphism in $\HH(\UU)$.

To find of an epimorphism $C \to A$ in $\HH(\UU)$ such that the composition $C \to A \to B[2]$ is zero, it thus suffices to find an epimorphism $C \to M$ in $\HH(\UU)$ such that $C \to M \to K_B[2]$ is zero.  The situation is now similar to the original setting (where the map $A \to B[2]$ has been replaced by a map $M \to K_B[2]$).  Recall, however, that $\Hom(S, K_B) = 0$.

Following Lemma \ref{lemma:Ext2Splittings}, finding an epimorphism $C \to M$ in $\HH(\UU)$ such that $C \to M \to K_B[2]$ is zero, is equivalent to finding a monomorphism $K_B \to L$ in $\HH(\UU)$ such that the composition $M \to K_B[2] \to L[2]$ is zero.

The second step of the proof is similar to the first step.  Consider the triangle
$$\Hom(P_S,M) \otimes P_S \to M \to C_M \to \Hom(P_S,M) \otimes P_S[1],$$
built on the evaluation morphism $\Hom(P_S,M) \otimes P_S \to M$, and the triangle
$$K_M[1] \to C_M \to Q_M \to K_M[2]$$
where $K_M \cong H^{-1}_\UU(C_M)$ and $Q_M \cong H^0_\UU(C_M)$.  Applying $H^0_{\UU}(-)$ to the first triangle yields the exact sequence
$$0 \to K_M \to \Hom(P_S,M) \otimes P_S \to M \to Q_M \to 0$$
in $\HH(\UU)$.  As before, we have $\Hom(P_S, Q_M)=0$ so that Proposition \ref{proposition:SimpleTop}(\ref{Number6}) implies that $\Hom(Q_M, S) = 0$.

Using that $\Hom_{\Db \AA}(P_S, K_B[2]) = 0$ (see Proposition \ref{proposition:SimpleTop}(\ref{Number2})), we see that the morphism $M \to K_B[2]$ factors through a morphism $C_M \to K_B[2]$.  Consider the following morphism of triangles:
$$\xymatrix{
K_M[1] \ar[r] \ar@{=}[d] & C_M \ar[r] \ar[d] & Q_M \ar[r] \ar[d] & K_M[2] \ar@{=}[d] \\
K_M[1] \ar[r] & K_B[2] \ar[r] & N[2] \ar[r] & K_M[2]
}$$
Here, the lower triangle corresponds to the short exact sequence $0 \to K_B \to N \to K_M \to 0$ in $\HH(\UU)$ (see Remark \ref{remark:ExtInHeart}), so that the map $K_B \to N$ is a monomorphism in $\HH(\UU)$.  We are looking for a monomorphism $K_B \to L$ in $\HH(\UU)$ such that the composition $M \to K_B[2] \to L[2]$ is zero, and for this it suffices that we find a monomorphism $N \to L$ such that $Q_M \to N[2] \to L[2]$ is zero.  Using Lemma \ref{lemma:Ext2Splittings} again, this is equivalent to showing that there is an object $X \in \HH(\UU)$ such that the map $Q_M \to N[2]$ factors as $Q_M \to X[1] \to N[2]$.

Recall that $\Hom(Q_M,S) = 0$.  We claim that $\Hom(S,N) = 0$.  In this case, the required factorization is then given by the assumptions in the statement of the lemma.

Applying $\Hom(S,-)$ to the short exact sequence $0 \to K_B \to N \to K_M \to 0$ in $\HH(\UU)$, shows that it is sufficient to prove that $\Hom(S,K_B) = 0 = \Hom(S,K_M)$.  We have already established that $\Hom(S,K_B) = 0$.  For the other equality, recall that there is a monomorphism $K_M \to \Hom(P_S,M) \otimes P_S$.  Seeking a contradiction, assume that $\Hom(S,K_M) \not= 0$, implying that $\Hom(S,P_S) \not= 0$.  There is then a nonzero composition $P_S \to S \to P_S$ which is invertible by Proposition \ref{proposition:HappelRingel}, yielding that $S \cong P_S$.  This implies that the composition $P_S \to K_M \to \Hom(P_S,M) \otimes P_S$ is a split monomorphism such that the composition $P_S \to \Hom(P_S,M) \otimes P_S \to M$ is zero, contradicting the universal property of the evaluation map (see Remark \ref{remark:Evaluation}).  This shows that $\Hom(S,K_M) = 0$ and finishes the proof.
\end{proof}

\begin{proposition}\label{proposition:Reduction}
Let $\Lambda$ be a finite-dimensional hereditary algebra, and let $(\UU, \VV)$ be a bounded $t$-structure in $\Db \mod \Lambda$.  Assume furthermore that $\bS\UU \subseteq \UU$.  Let $E \in \UU$ be an indecomposable $\UU$-projective object and let $S \in \HH(\UU)$ be the corresponding simple top (see Proposition \ref{proposition:SimpleTop}).  We write $\UU'$ for $\UU \cap {}^\perp S$.

If the aisle $\UU' \subseteq {}^\perp S$ induces a triangle equivalence $\Db \HH(\UU') \stackrel{\sim}{\rightarrow} {}^\perp S$, then the aisle $\UU \subseteq \Db \AA$ induces a triangle equivalence $\Db \HH(\UU) \stackrel{\sim}{\rightarrow} \Db \AA$.
\end{proposition}

\begin{proof}
By Lemma \ref{lemma:OnlyExt2}, we only need to check that every morphism $A \to B[2]$ (for $A,B \in \HH(\UU)$) factors as $A \to Z[1] \to B[2]$ (for some $Z \in \HH(\UU)$), and by Lemma \ref{lemma:Ext2Splittings}, this is equivalent to finding an epimorphism $C \to A$ in $\HH(\UU)$ such that $C \to A \to B[2]$ is zero.  By Lemma \ref{lemma:FurtherReduction}, we may furthermore assume that $\Hom(A,S) = 0 = \Hom(S,B)$. 

By Lemma \ref{lemma:RelatesHearts}, we know that $T^*_S(A), T_S^*(B) \in \HH(\UU')$, and thus we may assume (using Theorem \ref{theorem:BBD}) that there is an epimorphism $C' \to T^*_S(A)$ in $\HH(\UU')$ such that the composition $C' \to T^*_S(A) \to T^*_S(B)[2]$ is zero.  We get the following commutative diagram where the rows are triangles:
$$\xymatrix{C' \ar[r] \ar[d] & H^0_{\UU} (C') \ar[r] \ar[d] & \bigoplus_{i \geq 2} \Hom(S[i-1],C') \otimes S[i] \ar[r] \ar[d] & C'[1] \ar[d] \\
T_S^*(A) \ar[r] \ar[d]_{T_S^* f} & A \ar[r] \ar[d]_{f} & \RHom(A,S)^* \stackrel{L}{\otimes} S \ar[r] \ar[d]&  T_S^*(A)[1] \ar[d] \\
T_S^*(B)[2] \ar[r] & B[2] \ar[r] & \RHom(B[2],S)^* \stackrel{L}{\otimes} S \ar[r] &  T_S^*(B)[3]
}$$
where we have used Lemma \ref{lemma:RelatesHearts} to determine the topmost triangle.  We claim that $H^0_\UU (C') \to A$ is the required epimorphism $C \to A$ mentioned in the beginning of the proof.  We will write $C$ for $H^0_\UU (C')$.

First, we will show that $C \to A \to B[2]$ is zero.  The commutative diagram given above, yields the diagram
$$\xymatrix{C' \ar[r] \ar[d]_{0} & C \ar[r] \ar[d] & \bigoplus_{i \geq 2} \Hom(S[i-1],C') \otimes S[i] \ar[r] \ar[d] & C'[1] \ar[d] \\
T_S^*(B)[2] \ar[r] & B[2] \ar[r] & \RHom(B[2],S)^* \stackrel{L}{\otimes} S \ar[r] &  T_S^*(B)[3]
}$$
This implies that $C' \to C \to B[2]$ is zero, and hence the morphism $C \to B[2]$ factors as
$$C \to \bigoplus_{i \geq 2} \Hom(S[i-1],C') \otimes S[i] \stackrel{g}{\rightarrow} B[2].$$
Since $B \in \HH(\UU)$ and $S \in \UU$, we know that $\Hom(S[i], B[2]) = 0$ for $i \geq 3$.  Since we have assumed that $\Hom(S,B) = 0$, we know that $g$ is zero and thus so is the morphism $C \to B[2]$.

Next, we show that $C \to A$ is an epimorphism in $\HH(\UU)$.  For this, let $h:A \to H$ be a nonzero map in $\HH(\UU)$.  There is the following morphism of triangles:
$$\xymatrix{C' \ar[r] \ar[d] & C \ar[r] \ar[d] & \bigoplus_{i \geq 2} \Hom(S[i-1],C') \otimes S[i] \ar[r] \ar[d] & C'[1] \ar[d] \\
T_S^*(A) \ar[r] \ar[d]_{T_S^* h} & A \ar[r] \ar[d]_{h} & \RHom(A,S)^* \stackrel{L}{\otimes} S \ar[r] \ar[d]&  T_S^*(A)[1] \ar[d] \\
T_S^*(H) \ar[r] & H \ar[r] & \RHom(H,S)^* \stackrel{L}{\otimes} S \ar[r] &  T_S^*(H)[1]
}$$

Seeking a contradiction, we will assume that the composition $T_S^*(A) \to A \to H$ is zero.  In this case, $h:A \to H$ would factor as
$$A \to \RHom(A,S)^* \stackrel{L}{\otimes} S \to H.$$
Recall that $\RHom(A,S)^* \stackrel{L}{\otimes} S \cong \oplus_i \Hom(A,S[i])^* \otimes S[i]$.  Since $A \in \HH(\UU) \subseteq \UU$ and $S \in \HH(\UU)$, we know that $\Hom(A,S[i]) = 0$ for $i < 0$.  Furthermore, we have assumed that $\Hom(A,S) = 0$.  Combined, this shows that
$$\RHom(A,S)^* \stackrel{L}{\otimes} S \cong \oplus_{i > 0} \Hom(A,S[i])^* \otimes S[i].$$
However, since $S \in \UU$ and $H \in \HH(\UU)$, we know that $\Hom(S[i], H) = 0$ for $i > 0$, and thus we infer that the map $h: A \to H$ is zero.  This is a contradiction, hence we may assume that the composition $T_S^*(A) \to A \to H$ is nonzero.  We may also infer that the composition $T_S^*(A) \to T_S^*(H) \to H$ is nonzero, and since $C' \to T_S^*(A)$ is an epimorphism in $\HH(\UU')$, we find that the composition $C' \to T_S^*(A) \to T_S^*(H)$ is nonzero (here we have used Lemma \ref{lemma:RelatesHearts} to see that $T_S^*(H) \in \HH(\UU')$).

Using that $C' \in {}^\perp S$, we find that the composition $C' \to T_S^*(A) \to T_S^*(H) \to H$ is nonzero.  Hence the composition $C' \to C \to A \to H$ is nonzero, and thus so is $C \to A \to H$.  We conclude that $C \to A$ is indeed an epimorphism.

This finishes the proof.
\end{proof}
\section{Proof of the main theorem}\label{section:Proof}

We are now ready to prove the main theorem (Theorem \ref{theorem:MainTheorem} below).  Let $(\UU, \VV)$ be a $t$-structure on a triangulated category $\CC$ with Serre duality.  We showed in Corollary \ref{corollary:Needed} that if $(\UU, \VV)$ induces a triangle equivalence $\Db \HH(\UU) \to \CC$, then $(\UU, \VV)$ is bounded and $\bS \UU \subseteq \UU$.  In this section, we show the converse to this statement when $\CC$ is $\Db \mod \Lambda$ for a finite-dimensional hereditary algebra $\Lambda$.

\begin{theorem}\label{theorem:MainTheorem}
Let $\Lambda$ be a finite-dimensional hereditary algebra, and let $\bS$ be the Serre functor in $\Db \mod \Lambda$.  A $t$-structure $(\UU, \VV)$ on $\Db \mod \Lambda$ induces a triangle equivalence $\Db \HH(\UU) \stackrel{\sim}{\rightarrow} \Db \mod \Lambda$ if and only if $(\UU, \VV)$ is bounded and $\bS \UU \subseteq \UU$.
\end{theorem}

\begin{proof}
That the condition $\bS \UU \subseteq \UU$ is required, has been shown in Corollary \ref{corollary:Needed}.  For the other direction, let $(\UU, \VV)$ be a bounded $t$-structure on $\Db \mod \Lambda$.  If $\UU$ has no nonzero $\UU$-projective objects, then the statement follows from Corollary \ref{corollary:NoProjectivesInAisle}.

If $\UU$ does have nonzero $\UU$-projectives, then we will follow the strategy of \S\ref{section:Reduction}.  Let $P_0 \in \UU$ be an indecomposable $\UU$-projective and let $S_0 \in \HH(\UU)$ be the associated simple top.  By Proposition \ref{proposition:HappelRickardSchofield}, we know that ${}^\perp S_0$ is equivalent to $\Db \mod \Lambda'$, for some finite-dimensional hereditary algebra $\Lambda'$ with 1 fewer distinct simple module (thus $\mod \Lambda'$ has one fewer isomorphism class of simple objects than $\mod \Lambda$ has).

We can thus iterate this procedure, finding a sequence $S_0, S_1, \ldots, S_n$ of exceptional objects such that $$\UU_{n+1} = \UU \cap {}^\perp \{S_0, S_1, \ldots, S_n\}$$
has no nonzero $\UU_{n+1}$-projective objects.  Applying Proposition \ref{proposition:Reduction} $n+1$ times then yields the required result.
\end{proof}

\begin{example}
Let $\Lambda$ be a finite-dimensional hereditary algebra, and let $(\TT, \FF)$ be a torsion theory on $\mod \Lambda$.  Let $\BB$ be the tilting of $\mod \Lambda$ with respect to this torsion pair.  The natural embedding $\BB \to \Db \mod \Lambda$ induces a derived equivalence if and only if for every projective object $P \in \TT$, the corresponding injective $P \otimes_\Lambda \Lambda^*$ is also contained in $\TT$.

In particular, if $\TT$ has no nonzero projective objects (and thus the torsion theory is cotilting) or if $\TT$ has all injective objects (and thus the torsion theory is tilting), then $\BB$ is derived equivalent to $\mod \Lambda$.
\end{example}

\begin{example}
Let $\bX$ be a weighted projective line of domestic type (thus the category $\coh \bX$ is derived equivalent to $\mod \Lambda$ for a finite-dimensional hereditary algebra $\Lambda$).  Tilting with respect to any torsion pair $(\TT, \FF)$ on $\coh \bX$ induces a derived equivalence.
\end{example}

For a general (algebraic) triangulated category $\CC$, a bounded $t$-structure $(\UU, \VV)$ satisfying the condition $\bS \UU \subseteq \UU$ does not necessarily induce a triangle equivalence $\Db \HH(\UU) \to \CC$, as the following examples illustrate.

\begin{example}\label{example:NotInverse}
Let $\AA$ be the category $\grmod k[x]$ of finitely generated $\bZ$-graded $k[x]$-modules.  It is well-known that $\AA$ is hereditary and that the Serre functor on $\Db \AA$ is given by $\bS X \cong X(-1)[1]$.

For any $i \in \bZ$, We consider the full subcategories $\AA_{\leq i}$ and $\AA_{\geq i}$ of $\AA$ given by
\begin{eqnarray*}
\Ob \AA_{\leq i} &=& \{X \in \AA \mid \mbox{$X_j = 0$, for all $j > i$}\}, \\
\Ob \AA_{\geq i} &=& \{X \in \AA \mid \mbox{$X_j = 0$, for all $j < i$}\}.
\end{eqnarray*}
Note that the embedding $\AA_{\leq i} \to \AA$ has a right adjoint and that the embedding $\AA_{\geq i} \to \AA$ has a left adjoint; both adjoints are given by truncations.

We consider the following $t$-structure on $\Db \AA$:
\begin{eqnarray*}
\Ob D^{\leq 0} &=& \{X \in \Db \AA \mid H^{-i} X \in \AA_{\geq i}\}, \\
\Ob D^{\geq 0} &=& \{X \in \Db \AA \mid H^{-i} X \in \AA_{\leq i}\}.
\end{eqnarray*}
It is easily checked that $D^{\leq 0}$ is a preaisle, and it follows from \cite[Theorems 1.1 and 1.2]{StanleyvanRoomalen12} that $(D^{\leq 0}, D^{\geq 0})$ is a $t$-structure.

Furthermore, the heart $\HH = D^{\leq 0} \cap D^{\geq 0}$ can be described as follows: an object $X \in \Db \AA$ lies in $\HH$ if and only if $H^{-i}(X) \in \AA_{\geq i} \cap \AA_{\leq i}$, thus if $H^{-i} X$ is concentrated in degree $i$.  This means that the heart is the additive closure of $\bigcup_{i} \{k(i)[i]\}$ in $\Db \AA$, or thus equivalent to $\grmod k$; in particular, the heart is a semi-simple category and is thus not derived equivalent to $\AA$.
\end{example}

\begin{example}\label{example:NotInverse1}
Let $Q$ be the quiver $A_\infty$ with zig-zag orientation, thus $Q$ is the quiver:
$$\cdot \rightarrow \cdot \leftarrow \cdot \rightarrow \cdot \leftarrow \cdot \rightarrow \cdots$$
It is well-known that the category $\rep Q$ of finite-dimensional representations of $Q$ has Serre duality.  It follows from \cite{Keller05} that the orbit category $\DD \cong \Db \rep Q / (\bS [-2])$ is a triangulated category which admits a Serre functor $\bS_\DD \cong [2]$.  This orbit category $\DD$ has been discussed in \cite{HolmJorgensen12}.  Note since $\bS_\DD \cong [2]$, we have $\bS_\DD \UU \subseteq \UU$ for every aisle $\UU \subseteq \DD$.

The $t$-structures in $\DD$ have been classified in \cite[Theorem 4.1]{Ng10}, and it follows from that classification that for every bounded aisle $\UU \subseteq \DD$, the heart $\HH(\UU) \cong \mod k$, and hence $\DD \not\cong \Db \HH(\UU)$.
\end{example}

\begin{example}\label{example:NotInverse2}
Let $A$ be the dg-algebra $k[t]$ with zero differential and $\deg t = -n$ for some $n \in \bZ$.  We will write $A^e = A \otimes_k A^{op}$, thus $A^e \cong k[s,t]$ with zero differential and $\deg t = \deg s = -n$.

The dg-algebra $A$ is homologically smooth, meaning that, as an $A^e$-module, it has a finite resolution by finitely generated projective objects.  Here, such a resolution is given by:
$$\xymatrix@1{ \cdots \ar[r]& 0  \ar[r] & k[s,t](n) \ar[r]^-{s-t} & k[s,t] \ar[r] & 0 \ar[r] & \cdots}$$
where $k[s,t](n)$ is the dg-algebra $k[s,t]$ shifted by degree $n$ so that multiplication by $s-t$ is a degree zero morphism.

We will write $\Theta_A$ for a cofibrant replacement (as dg $A^e$-modules) of $\RHom_{A^e}(A,A^e)$, thus the associated complex is
$$\xymatrix@1{ \cdots \ar[r]& 0  \ar[r] & k[s,t] \ar[r]^-{s-t} & k[s,t](-n) \ar[r] & 0 \ar[r] & \cdots}$$
and we obtain the dg $A^e$-module $\Theta_A$ by taking the total complex of the associated bicomplex.  It follows from \cite[Lemma 3.4]{Keller11} that the derived category $D_\text{fd}(A)$ of finite-dimensional right dg-modules is $(n+1)$-Calabi-Yau, thus the Serre functor $\bS: D_\text{fd}(A) \to D_\text{fd}(A)$ is given by $[n+1]$.

From now on, assume that $n > 0$, thus $A$ is a nonpositively graded dg-algebra with finite-dimensional cohomologies.  Since $\bS \cong [n+1]$, every aisle in $D_\text{fd}(A)$ is closed under the Serre functor.

Let $(\UU, \VV)$ be the standard $t$-structure on $D_\text{fd}(A)$, thus $M \in \UU$ if and only if $H^{i}(M) = 0$ for all $i > 0$.  We remarked before that every aisle in $D_\text{fd}(A)$ is closed under the Serre functor, thus $\bS \UU \subseteq \UU$.  Following \cite[\S4.1]{KonigYang14}, we have $\HH(\UU) \cong \mod k$, which shows that $\Db \HH(\UU) \not\cong D_\text{fd}(A)$.
\end{example}

\begin{remark}
The category $D_\text{fd}(A)$ from Example \ref{example:NotInverse2} has been described in more detail in \cite[Section 8]{Jorgensen04}.
\end{remark}

\def\cprime{$'$}
\providecommand{\bysame}{\leavevmode\hbox to3em{\hrulefill}\thinspace}
\providecommand{\MR}{\relax\ifhmode\unskip\space\fi MR }
\providecommand{\MRhref}[2]{%
  \href{http://www.ams.org/mathscinet-getitem?mr=#1}{#2}
}
\providecommand{\href}[2]{#2}

\end{document}